\documentclass[11pt]{article}
\usepackage{latexsym,amsfonts,amssymb,amsmath,amsthm}
\usepackage{graphicx}

\usepackage[usenames,dvipsnames]{color}
\usepackage{ulem,comment}

\usepackage{color}

\title{Can chemotaxis speed up or slow down the spatial spreading in parabolic-elliptic Keller-Segel  systems with logistic source?}

\author{
Rachidi B. Salako\\
Department of Mathematics\\
The Ohio State University\\
Columbus OH, 43210-1174\\
\\
 and \\
 \\
 Wenxian Shen\thanks{Partially supported by the NSF grant DMS--1645673} \,\, and\,\,  Shuwen Xue\\
Department of Mathematics and Statistics\\
Auburn University\\
Auburn University, AL 36849 }

\date{}

\begin{document}

\maketitle

\newtheorem{tm}{Theorem}[section]
\newtheorem{prop}{Proposition}[section]
\newtheorem{defin}{Definition}[section] 
\newtheorem{coro}{Corollary}[section]
\newtheorem{lem}{Lemma}[section]
\newtheorem{assumption}{Assumption}[section]
\newtheorem{rk}{Remark}[section]
\newtheorem{nota}[tm]{Notation}
\numberwithin{equation}{section}

\newcommand{\stk}[2]{\stackrel{#1}{#2}}
\newcommand{\dwn}[1]{{\scriptstyle #1}\downarrow}
\newcommand{\upa}[1]{{\scriptstyle #1}\uparrow}
\newcommand{\nea}[1]{{\scriptstyle #1}\nearrow}
\newcommand{\sea}[1]{\searrow {\scriptstyle #1}}
\newcommand{\csti}[3]{(#1+1) (#2)^{1/ (#1+1)} (#1)^{- #1
 / (#1+1)} (#3)^{ #1 / (#1 +1)}}
\newcommand{\RR}[1]{\mathbb{#1}}

\newcommand{\rd}{{\mathbb R^d}}
\newcommand{\ep}{\varepsilon}
\newcommand{\rr}{{\mathbb R}}
\newcommand{\alert}[1]{\fbox{#1}}
\newcommand{\eqd}{\sim}
\def\p{\partial}
\def\R{{\mathbb R}}
\def\N{{\mathbb N}}
\def\Q{{\mathbb Q}}
\def\C{{\mathbb C}}
\def\l{{\langle}}
\def\r{\rangle}
\def\t{\tau}
\def\k{\kappa}
\def\a{\alpha}
\def\la{\lambda}
\def\De{\Delta}
\def\de{\delta}
\def\ga{\gamma}
\def\Ga{\Gamma}
\def\ep{\varepsilon}
\def\eps{\varepsilon}
\def\si{\sigma}
\def\Re {{\rm Re}\,}
\def\Im {{\rm Im}\,}
\def\E{{\mathbb E}}
\def\P{{\mathbb P}}
\def\Z{{\mathbb Z}}
\def\D{{\mathbb D}}
\newcommand{\ceil}[1]{\lceil{#1}\rceil}

\begin{abstract}
The current paper is concerned with the spatial spreading speed and minimal wave speed  of the
 following Keller-Segel chemoattraction system,
\begin{equation}\label{abstract-eq1}
\begin{cases}
u_t=u_{xx}-\chi(uv_x)_x +u(a-bu),\quad x\in\R\cr
0=v_{xx}- \lambda v+\mu u,\quad x\in\R,
\end{cases}
\end{equation}
where  $\chi$, $a$, $b$, $\lambda$, and $\mu$ are positive constants.
 Assume $b>\chi\mu$. Then if in addition, $\big(1+\frac{1}{2}\frac{(\sqrt{a}-\sqrt{\lambda})_+}{(\sqrt{a}+\sqrt{\la})}\big)\chi\mu { \leq} b$ holds, it is proved that $c_0^*=2\sqrt a$ is the spreading speed of the solutions of  \eqref{abstract-eq1} with nonnegative continuous initial function $u_0$ with nonempty compact support, that is,
 $$
 \limsup_{|x|\ge ct, t\to\infty}u(t,x;u_0)=0\quad \forall\, c>c_0^*
 $$
 and
 $$
 \liminf_{|x|\le ct,t\to\infty} u(t,x;u_0)>0\quad \forall \, 0<c<c_0^*,
 $$
 where $(u(t,x;u_0),v(t,x;u_0))$ is the unique global classical solution of \eqref{abstract-eq1} with $u(0,x;u_0)=u_0(x)$.
 It is also proved that,  if  $b>2\chi\mu$  and $\lambda \geq a$ holds, then  $c_0^*=2\sqrt a$ is the minimal speed  of the  traveling wave solutions of \eqref{abstract-eq1}  connecting $(0,0)$ and $(\frac{a}{b},\frac{\mu}{\lambda}\frac{a}{b})$, that is,
  for any $c\ge c_0^*$, \eqref{abstract-eq1}
 has a  traveling wave solution  connecting $(0,0)$ and $(\frac{a}{b},\frac{\mu}{\lambda}\frac{a}{b})$ with speed $c$,
 and \eqref{abstract-eq1} has no such traveling wave solutions with speed less than $c_0^*$. Note that $c_0^*=2\sqrt a$ is the spatial spreading speed as well as the minimal wave speed of the following Fisher-KPP equation,
 \begin{equation}
 \label{abstract-eq2}
 u_t=u_{xx}+u(a-bu),\quad x\in\R.
 \end{equation}
 Hence, if $\lambda \geq a$ and $b>\chi\mu$, or $\lambda<a$ and  $b\geq \big(1+\frac{1}{2}\frac{(\sqrt{a}-\sqrt{\lambda})}{(\sqrt{a}+\sqrt{\la})}\big)\chi\mu$   then the chemotaxis neither speeds up nor slows down the spatial spreading in \eqref{abstract-eq1}.

\end{abstract}

\medskip
\noindent{\bf Key words.} Parabolic-elliptic chemotaxis system, logistic source, classical solution, spreading speeds, traveling waves.

\medskip
\noindent {\bf 2010 Mathematics Subject Classification.} 35B35, 35B40, 35K57, 35Q92, 92C17.

\section{Introduction}

This work is concerned with the propagation speeds of solutions in the attraction Keller-Segel chemotaxis models of the form
\begin{equation}\label{Keller-Segel-eq0}
\begin{cases}
u_t=\Delta u-\nabla \cdot (\chi u \nabla v)+u(a-bu),\quad x\in\R^N\cr
0 =\Delta v-  \lambda v +\mu u,\quad x\in\R^N,
\end{cases}
\end{equation}
where $a, b, \lambda,\mu$ and $\chi>0$ are positive constants,  and $u(t,x)$ and $v(t,x)$ represent the densities of the
 mobile species and the chemo-attractant, respectively.  Biologically, the positive constant $\chi$ measures the sensitivity effect on the mobile species by  the chemical substance which is produced overtime by the mobile species;   the reaction $u(a-bu)$ in the first  equation of \eqref{Keller-Segel-eq0} describes the local dynamics of the mobile species; $\lambda$  represents the degradation rate of the  chemo-attractant;  and $\mu$ is the rate at which the mobile species produces the chemo-attractant.

  System \eqref{Keller-Segel-eq0} is a simplified version of the chemotaxis system proposed by Keller and Segel in their works \cite{KeSe1,KeSe2}.
  Chemotaxis models describe the oriented movements of biological cells and organisms in response to chemical gradient which they may produce themselves  over time. These mathematical models play very important roles in a wide range of biological phenomena and accordingly a considerable literature is concerned with its mathematical analysis.  The reader is referred to \cite{HiPa, Hor}  for some detailed introduction into the mathematics of Keller-Segel models.

{  A famous application of chemotaxis models is to describe the life cycle of Dictyostelium discoideum. As described in \cite{AM_AK_LE},  D. discoideum lives in the soil and feeds on bacteria and other microorganisms that are taken up by phagocytosis. During the vegetative growth stage, the single-celled amoebae divide by simple mitotic divisions. In times of starvation, a developmental program is initiated, which is accompanied by major changes
in gene expression. As a result, cells
begin to signal each other by secreting cAMP and to aggregate by
chemotaxis toward this chemoattractant. The resulting multicellular aggregate contains up to a few hundred thousand cells and undergoes further differentiation and morphogenetic changes. Finally a fruiting body is formed which consists of two main cell types, spore and stalk cells. The stalk consists of dead vacuolated cells, while the spore cells are resistant to
extreme temperatures or drought. More favorable environmental conditions enable the hatching of new amoebae from the spores. The aggregation of thousands of individual cells that build a multicellular organism in this peculiar life cycle, has intrigued scientists for decades.

}

  The study of the dynamics of solutions to \eqref{Keller-Segel-eq0} has attracted a number of researchers over the past few years.  Finite time blow-up phenomena is among important dynamical issues about \eqref{Keller-Segel-eq0}.  This phenomena has been studied in many  papers in the case $a=b=0$ (see \cite{HiPa,  Dirk and Winkler, KKAS, kuto_PHYSD, NAGAI_SENBA_YOSHIDA,  win_jde, win_JMAA_veryweak, win_arxiv}).   It is shown that finite time blow-up  may occur if either $N=2$ and the total initial population mass is large enough, or $N\geq 3$. It is also shown that some radial solutions to \eqref{Keller-Segel-eq0} in plane collapse into a persistent Dirac-type singularity in the sense that a globally defined measure-valued solution exists which has a singular part beyond some finite time and asymptotically approaches a Dirac measure (see \cite{LSV1,TeWi1}). We refer the reader to \cite{BBTW,HoSt} and the references therein for more insights in the studies of chemotaxis models.

When the constant $a$ and $b$ are positive, the finite time blow-up phenomena in \eqref{Keller-Segel-eq0} may be suppressed  to some  extent. In fact in this case, it is known that when the space dimension is equal to one or two, solutions to \eqref{Keller-Segel-eq0}  with initial functions in a space of certain integrable functions are defined for all time. And it is enough for the self limitation coefficient $b$ to be big enough comparing  to the chemotaxis sensitivity coefficient to prevent finite time blow-up, see \cite{ITBWS16,SaSh1,TeWi2}.

 Spatial spreading dynamics  is another important dynamical issue about  \eqref{Keller-Segel-eq0}.
Observe that, when $\chi=0$, the chemotaxis system \eqref{Keller-Segel-eq0} reduces to
\begin{equation}\label{fisher-kpp}
u_t=\Delta u+u(a-bu), \quad x\in\R^N.
\end{equation}
Due to the pioneering works of Fisher
\cite{Fisher} and Kolmogorov, Petrowsky, Piskunov \cite{KPP} on traveling wave solutions and take-over properties of \eqref{fisher-kpp},
\eqref{fisher-kpp} is also referred to as  the Fisher-KPP equation.
 The following results are well known about the spatial spreading dynamics of \eqref{fisher-kpp}.
Equation \eqref{fisher-kpp} has  traveling wave solutions $u(t,x)=\phi(x\cdot\xi-ct)$ ($\xi\in S^{N-1}$)
connecting $\frac{a}{b}$ and $0$ $(\phi(-\infty)=\frac{a}{b},\phi(\infty)=0)$ of all speeds $c\geq 2\sqrt a$ and
has no such traveling wave solutions of slower
speed.  For any
nonnegative solution $u(t,x)$ of (\ref{fisher-kpp}), if at
time $t=0$, $u(0,x)=u_0(x\cdot\xi)$ $(\xi\in S^{N-1})$ is $\frac{a}{b}$ for $x\cdot\xi$ near $-\infty$ and $0$ for $x\cdot\xi$ near $ \infty$, then
$$\limsup_{x\cdot\xi \ge ct, t\to \infty}u(t,x)=0 \quad \forall \, c>2\sqrt a
$$
and
$$\limsup_{x\cdot\xi \le ct, t\to \infty}|u(t,x)-\frac{a}{b}|=0\quad \forall\,  c<2\sqrt a.
$$
In
literature, $c^*_0=2\sqrt a$ is   called the {\it
spreading speed} for \eqref{fisher-kpp}.  Since the pioneering works by  Fisher \cite{Fisher} and Kolmogorov, Petrowsky,
Piscunov \cite{KPP},  a huge amount of research has been carried out toward the front propagation dynamics of
  reaction diffusion equations of the form,
\begin{equation}
\label{general-fisher-eq}
u_t=\Delta u+u f(t,x,u),\quad x\in\R^N,
\end{equation}
where $f(t,x,u)<0$ for $u\gg 1$,  $\partial_u f(t,x,u)<0$ for $u\ge 0$ (see \cite{ArWe2, BHN,  BeHaNa1, BeHaNa2, Henri1, Fre, FrGa, LiZh, LiZh1, Nad, NoRuXi, NoXi1, She1, She2, Wei1, Wei2, Zla}, etc.).

 Recently, the first two authors of the current paper studied the spatial spreading dynamics of \eqref{Keller-Segel-eq0}
 and obtained several fundamental results. Some lower and upper bounds for the propagation speeds of solutions with compactly supported initial functions were derived, and some lower bound for the speeds of traveling wave solutions was also derived. It is proved that all these bounds converge to the spreading speed $c_0^*=2\sqrt a$ of \eqref{fisher-kpp} as $\chi\to 0$  (see \cite{SaSh3}, \cite{SaSh2}, \cite{SaSh1}).
 The reader is also referred to  \cite{FhCh} for the lower and upper bounds of propagation speeds of \eqref{Keller-Segel-eq0},
  and is referred to 
\cite{AiHuWa, AiWa, FuMiTs, HoSt, LiLiWa, MaNoSh,
NaPeRy,Wan}, etc.,  for the  studies on traveling
wave solutions of various types of chemotaxis models.

  However, several important biological and mathematical problems remain open. For example,
whether the presence of the chemical substance in \eqref{Keller-Segel-eq0}  slows down or speeds up the propagation of mobile species,
 and whether  there is a minimal wave speed of \eqref{Keller-Segel-eq0}.
  It is the aim of the current paper to provide answers to these questions for some range of the parameters $a,b,\lambda,\mu$ and $\chi$. We remark that,  to study the spatial spreading speeds and traveling wave solutions of \eqref{Keller-Segel-eq0} along some direction
  $\xi\in\ S^{N-1}$ (i.e. study solutions of the form $u(t,x)=\tilde u(t,x\cdot\xi)$), it suffices to study these issues for \eqref{Keller-Segel-eq0} with $N=1$, that is,
  \begin{equation}\label{Keller-Segel-eq}
\begin{cases}
u_t=u_{xx}-\chi(uv_x)_x +u(a-bu),\quad x\in\R\cr
0=v_{xx}- \lambda v+\mu u,\quad x\in\R.
\end{cases}
\end{equation}
 In the rest of this introduction, we state the main results on the spatial spreading dynamics of \eqref{Keller-Segel-eq}.

\subsection{Statement of  the  main results.}

In this subsection, we state the main results of the paper. In order to do so, we first introduce some notations and definitions. Let
$$
C^b_{\rm unif}(\R)=\{ u \, |\, \R\to \R  \ u\ \text{is uniformly continuous and bounded}\}.
$$
For every $u\in C^b_{\rm unif}(\R)$ we let $\|u\|_{\infty}:=\sup_{x\in\R}|u(x)|$. For each given  $u_0\in C^b_{\rm unif}(\R)$  with $u_0(x)\geq 0$, we denote by $(u(t,x;u_0),v(t,x;u_0))$ the classical solution of \eqref{Keller-Segel-eq} satisfying $u(0,x;u_0)=u_0(x)$ for every $x\in\R$. Note that, by comparison principle for parabolic equations, for every  nonnegative initial function $u_0\in C^b_{\rm unif}(\R)$, it always holds that $u(t,x;u_0)\geq 0$ and $v(t,x;u_0)\geq 0$ whenever $(u(t,x;u_0),v(t,x;u_0))$ is defined. In this work we shall only focus on nonnegative classical solutions of \eqref{Keller-Segel-eq}  since both functions $u(t,x)$ and $v(t,x)$ represent density functions. We recall the following result proved in \cite{SaSh1}.

\begin{prop}
For every nonnegative initial function $u_0\in C^b_{\rm unif}(\R)$, there is a unique maximal time $T_{max}$, such that $(u(t,x;u_0),v(t,x;u_0))$ is defined for every $x\in\R$ and $0\le t<T_{\max}$ (\cite[Theorem 1.1]{SaSh1}). Moreover if $\chi\mu<b$ then $T_{max}=\infty$ (\cite[Theorem 1.5]{SaSh1}).
\end{prop}

To state our main result on the spreading speeds of solutions of \eqref{Keller-Segel-eq} with nonempty and compact supported initial functions, we first introduce the concept of spreading speeds.

 Suppose that $b>\chi\mu$. Let
$$
C_c^+(\R)=\{u\in C_{\rm unif}^b(\R)\,|\, u(x)\ge 0,\,\, {\rm supp}(u)\,\,\, \text{is non-empty and compact}\}.
$$
Let
$$
C_{-}^*=\{c^*_->0\,|\,
  \liminf_{t\to\infty} \inf_{|x|\le ct}  u(x,t;u_0)>0\quad \forall\,\, u_0\in C_c^+(\R),\,\, \forall\,  0<c<c_{-}^*\}
 $$
 and
 $$
 C_+^*=\{c ^*_+>0\,|\,
 \lim_{t\to\infty}\sup _{|x|\ge ct}   u(x,t;u_0)=0\quad \forall\,\, u_0\in C_c^+(\R),\,\, \forall\, c>c_{+}^*\}.
 $$
 Let
 $$
 c_{-}^*(\chi,a,b,\lambda,\mu)=\sup\{c\in C_-^*\}\quad {\rm and}\quad c_+^*(\chi,a,b,\lambda,\mu)=\inf\{c\in C_+^*\},
 $$
 where $c_-^*(\chi,a,b,\lambda,\mu)=0$ if $C_-^*=\emptyset$ and $c_+^*(\chi,a,b,\lambda,\mu)=\infty$ if $C_+^*=\emptyset$.
  It is clear that
$$
0\le c_-^*(\chi,a,b,\lambda,\mu)\le c_+^*(\chi,a,b,\lambda,\mu)\le\infty.
$$
 Thanks to the feature of $c_-^*:=c_-^*(\chi,a,b,\lambda,\mu)$ and $c_+^*:=c_+^*(\chi,a,b,\lambda,\mu)$, we call the interval $[c_-^*,c_+^*]$
the {\it spreading speed interval}   of solutions of  \eqref{Keller-Segel-eq}   with compactly supported initials.

\smallskip

 Let {\bf (H)} be the following standing assumption.

\smallskip

\noindent {\bf (H)}  $\big(1+\frac{1}{2}\frac{(\sqrt{a}-\sqrt{\lambda})_+}{(\sqrt{a}+\sqrt{\la})}\big)\chi\mu \leq b$.

\smallskip

Let $$
a^{*}=\max\Big\{\kappa\,\big |\, 0<\kappa\leq \sqrt{a},  \,\, \frac{(\kappa-\sqrt{\lambda})_{+}}{(\kappa +\sqrt{\lambda})}\leq \frac{2(b-\chi\mu)}{\chi\mu} \Big\}.
$$
and
$$
c^*=\frac{a+(a^*)^2}{a^*},\quad c_0^*=2\sqrt a.
$$
Observe that
$$
c^*\ge c_0^*,
$$
and that, if {\bf(H)} holds, then
$$
a^*= \sqrt a
$$
and hence
$$
c^*=2\sqrt a=c_0^*.
$$

We prove the following theorem on the upper and lower bounds of the spreading speed interval of \eqref{Keller-Segel-eq}.

\begin{tm}\label{spreading-speed-upper-bound}
Suppose that $0<\chi\mu<b$.   Then
\begin{itemize}
\item[(1)]
$$
c_+^*(\chi,a,b,\lambda,\mu)\le c^*.$$
In particular,  if {\bf (H)} holds, then
$$
 c_+^*(\chi,a,b,\lambda,\mu)\le c_0^*(=2\sqrt a).
 $$

\item[(2)]
\begin{equation}\label{T1-eq1}
c_-^*(\chi,a,b,\lambda,\mu)\ge c_0^*(=2\sqrt a).
\end{equation}
Moreover, if $ 2\chi\mu<b$, then
\begin{equation}\label{T1-eq2}
\lim_{t\to\infty}\sup_{|x|\leq  ct}|u(t,x;u_0)-\frac{a}{b}|=0\quad \forall\, c<c_0^*(=2\sqrt a).
\end{equation}
\end{itemize}
 \end{tm}

\begin{rk}
Assume that $\chi\mu<b$.

(1) Theorem \ref{spreading-speed-upper-bound} provides an upper bound and a low bound for $c_+^*(\chi,a,b,\lambda,\mu)$ and
$c_-^*(\chi,a,b,\lambda,\mu)$, respectively. As it is recalled in the above, in the absence of chemotaxis (i.e. $\chi=0$), we have
$$
c_+^*(\chi,a,b,\lambda,\mu)=c_-^*(\chi,a,b,\lambda,\mu)=c_0^*(=2\sqrt a).
$$
 Theorem \ref{spreading-speed-upper-bound} (2) shows that the chemotaxis does not slow down the spreading speed
of the solutions with  nonempty compactly supported initials, and Theorem \ref{spreading-speed-upper-bound} (1) shows that,
when
$\big(1+\frac{1}{2}\frac{(\sqrt{a}-\sqrt{\lambda})_+}{(\sqrt{a}+\sqrt{\la})}\big)\chi\mu \leq b$,
the chemotaxis does not speed up the spreading speed
of the solutions with  nonempty compactly supported initials. We note that $\lambda\ge a$ and $\chi\mu<b$ implies that {\bf(H)} holds. 
Biologically,
$\lambda\ge a$ means that the degradation rate of the chemo-attractant is greater than or equal to the intrinsic growth rate  of the mobile species, and
 $\big(1+\frac{1}{2}\frac{(\sqrt{a}-\sqrt{\lambda})_+}{(\sqrt{a}+\sqrt{\la})}\big)\chi\mu \leq b$ indicates
  the chemotaxis sensitivity is small relative to the logistic damping.

(2)  In \cite{SaSh1,SaSh2}, the first two authors of the current paper obtained  some constants  $c_{\rm low}^{*}(\chi,\mu,a,b,\lambda,\mu)<2\sqrt a< c_{\rm up}^*(\chi,a,b,\lambda,\mu)$ depending explicitly on the parameter $\chi,a,b,\lambda$ and $\mu$ such that
$$
 c_+^*(\chi,a,b,\lambda,\mu)\le  c_{\rm up}^*(\chi,a,b,\lambda,\mu)
$$
and
$$
c_{-}^{*}(\chi,\mu,a,b,\lambda,\mu)\ge c_{\rm low}^*(\chi,a,b,\lambda,\mu).
$$
 There holds
 $$
 c_{\rm low}^*(\chi,a,b,\lambda,\mu)\le  c_0^*\le c^*\le c_{\rm up}^*(\chi,a,b,\lambda,\mu).
 $$
 Hence Theorem \ref{spreading-speed-upper-bound} is an improvement of the results contained in \cite{SaSh1,SaSh2} on the lower and
 upper bounds for the spreading speeds of solutions with  nonempty compactly supported initials.

 (3)  The results in Theorem \ref{spreading-speed-upper-bound}(1) are new.
  Theorem \ref{spreading-speed-upper-bound}(2)  is proved using the similar arguments as those in \cite[Theorem 1.1]{FhCh}.
 Actually, in  the case $a=b=1$ and $\lambda=\mu$,   Theorem \ref{spreading-speed-upper-bound}(2) is proved
  in \cite[Theorem 1.1]{FhCh}.  The results in Theorem \ref{spreading-speed-upper-bound}(2) for the general case are new.
\end{rk}

The techniques developed to prove the above results can be used to study the spreading speeds of solutions with front like initials. Indeed,  let
$$
\tilde C_c^+(\R)=\{u\in C_{\rm unif}^b(\R)\,|\, u(x)\ge 0,\,\, \liminf_{x\to-\infty}u(x)>0\,\,\, \text{ and}\,\, u(x)=0 \ \text{for}\, x\gg 0\}.
$$
We can establish the following result.

\begin{tm}\label{spreading-speed-upper-bound-of-front-like-initials}
Suppose that $0<\chi\mu<b$.   Then
\begin{itemize}
\item[(1)]  For any $u_0\in\tilde C_c^+(\R)$, there holds
$$
\lim_{t\to\infty}\sup _{x\ge ct}  u(x,t;u_0)=0\quad  \forall\, c>c^*.
$$
In particular,  if {\bf (H)} holds, then for any $u_0\in\tilde C_c^+(\R)$, there holds
$$
\lim_{t\to\infty}\sup _{x\ge ct}  u(x,t;u_0)=0\quad  \forall\, c>c_0^*(=2\sqrt a).
$$

\item[(2)] For any $u_0\in \tilde C_c^+(\R)$, there holds
$$
\liminf_{t\to\infty}\inf_{x\le ct} u(x,t;u_0)>0\quad  \forall\,  0<c<c_0^*(=2\sqrt a).
$$
Moreover, if $ 2\chi\mu<b$, then for any $u_0\in\tilde C_c^+(\R)$,
\begin{equation}\label{T3-eq2}
\lim_{t\to\infty}\sup_{x\leq  ct}|u(t,x;u_0)-\frac{a}{b}|=0\quad \forall\, 0<c<c_0^*(=2\sqrt a).
\end{equation}
\end{itemize}
\end{tm}

\medskip

We also discuss the spreading properties of solutions of \eqref{Keller-Segel-eq} with initial functions satisfying some exponential decay 
property at infinity. In this direction, we have the following result.

\medskip

\begin{tm}\label{Main-tm-2}
Suppose that $0<\chi\mu<b$.
\begin{itemize}
\item[(1)] If $u_0\in C^b_{\rm unif}(\R)$ satisfies  that
\begin{equation}\label{u-0-condition-1}
\inf_{x\le x_0}u_{0}(x)>0 \quad \forall \,  x_0\in\R, \quad\text{and}\quad \lim_{x\to\infty}\frac{u_0(x)}{e^{-\kappa x}}=1
\end{equation}
for some $ 0<\kappa<\sqrt{a}$ with $\frac{(\kappa-\sqrt{\lambda})_{+}}{(\kappa +\sqrt{\lambda})}\leq \frac{2(b-\chi\mu)}{\chi\mu}$, then
\begin{equation}\label{Spreadind-speed-eq1}
\lim_{t\to\infty}\sup_{x\geq (c_{\kappa}+\varepsilon)t}|u(t,x;u_0)|=0, \quad\forall\ 0<\varepsilon\ll 1
\end{equation}
and
\begin{equation}\label{Spreadind-speed-eq1-1}
\liminf_{t\to\infty}\inf_{x\leq (c_{\kappa}-\varepsilon)t}|u(t,x;u_0)|>0, \quad\forall\ 0<\varepsilon\ll 1,
\end{equation}
where $c_{\kappa}=\frac{a+\kappa^2}{\kappa}$.

\item[(2)] Let $u_0$ be as in (1). If  $b>2\chi\mu$,  then
\begin{equation}\label{Spreadind-speed-eq2}
\lim_{t\to\infty}\sup_{x\leq (c_{\kappa}-\varepsilon)t}|u(t,x;u_0)-\frac{a}{b}|=0, \quad\forall\ 0<\varepsilon\ll 1,
\end{equation}
and,  if in addition, $\kappa<\min\{\sqrt a, \sqrt \lambda\}$, then
\begin{equation}\label{Spreadind-speed-eq3}
\lim_{t\to\infty}\sup_{x\geq (c_{\kappa}+\varepsilon)t}\left|\frac{u(t,x;u_0)}{e^{-\kappa( x-c_\kappa t)}}-1\right|=0, \quad\forall\ 0<\varepsilon\ll 1.
\end{equation}
\end{itemize}
\end{tm}

\begin{rk}
The spreading results  established in Theorem  \ref{Main-tm-2} are new.
\end{rk}

To state our main results on traveling wave solutions, we first introduce the concept of traveling wave solutions.
{\it An entire solution} of \eqref{Keller-Segel-eq} is a classical solution  $(u(t,x),v(t,x))$ of \eqref{Keller-Segel-eq} which is defined for all $x\in\R$ and $t\in\R$. Note that the constant solutions $(u(t,x),v(t,x))=(0,0)$ and  $(u(t,x),v(t,x))=(\frac{a}{b},\frac{\mu a}{\lambda b})$ are clearly two particular entire solutions of \eqref{Keller-Segel-eq}.  An entire solution of \eqref{Keller-Segel-eq} of the form $(u(t,x),v(t,x))=(U^c(x-ct),V^c(x-ct))$ for some constant $c\in\R$ is called a {\it traveling wave solution} with speed $c$.
 A traveling wave solution $(u(t,x),v(t,x))=(U^c(x-ct),V^c(x-ct))$ of \eqref{Keller-Segel-eq} with speed c is said to connect $(0,0)$ and $(\frac{a}{b},\frac{\mu a}{\lambda b})$  if
\begin{equation}\label{Main-TW-eq}
\liminf_{x\to-\infty}U^c(x)=\frac{a}{b} \quad \text{and}\quad \limsup_{x\to\infty}U^c(x)=0.
\end{equation}
 We say a traveling wave solution $(u(t,x),v(t,x))=(U^c(x-ct),V^c(x-ct))$ of \eqref{Keller-Segel-eq} is nontrivial and connects $(0,0)$ at one end if
\begin{equation}\label{Persisence-TW-eq}
\liminf_{x\to-\infty}U^c(x)>0 \quad \text{and}\quad \limsup_{x\to\infty.}U^c(x)=0.
\end{equation}

Our main results on the existence of traveling wave solutions of \eqref{Keller-Segel-eq} read as follows.

\begin{tm}\label{Existence of traveling} Suppose that $\chi\mu<b$.
\begin{itemize}
\item[(1)] For every $0<\kappa<\min\{\sqrt{a},\sqrt{\lambda}\}$, \eqref{Keller-Segel-eq} has a  traveling wave solution $(u(t,x),v(t,x))=(U(x-c_{\kappa}t),V(x-c_{\kappa}t))$ satisfying
\begin{equation}
\label{nnew-eq0}
\lim_{x\to\infty}\frac{U(x)}{e^{-\kappa x}}=1 \quad \text{and}\quad
\liminf_{x\to-\infty}U(x)>0,
\end{equation}
where $c_{\kappa}=\frac{\kappa^2+a}{\kappa}$. Hence \eqref{Keller-Segel-eq} has a  traveling wave solution satisfying  \eqref{nnew-eq0} with speed $c$ for every $c>c^{**}:=\frac{a+\min\{a,\lambda\}}{\min\{\sqrt{a},\sqrt{\lambda}\}}$.

 If, in addition,  $b>2\chi\mu$, then   $U(x)$ also satisfies
\begin{equation}
\label{nnew-eq00}
\lim_{x\to-\infty}|U(x)-\frac{a}{b}|=0.
\end{equation}

\item[(2)]   If $b>2\chi\mu$,
then \eqref{Keller-Segel-eq} has a traveling wave solution with
speed $c= c^{**}$ connecting $(0,0)$ and $(\frac{a}{b},\frac{\mu a}{\lambda b})$.

\item[(3)] \eqref{Keller-Segel-eq} has no traveling wave solutions satisfying  \eqref{Persisence-TW-eq} with
speed $c<c_0^*=2\sqrt{a}$.
\end{itemize}
\end{tm}

\begin{rk}
(1) It is known that in the absence of chemotaxis (i.e. $\chi=0$), $c_0^*=2\sqrt a$ is the minimal wave speed of \eqref{fisher-kpp}
in the sense that for any $c\ge c_0^*$, \eqref{fisher-kpp} has a traveling wave solution connecting $\frac{a}{b}$ and $0$ with speed $c$, and
has no such traveling wave solutions with speed less than $c_0^*$. Theorem \ref{Existence of traveling} implies that,
when  $b>2\chi\mu$, and $\lambda\ge a$, $c_0^*=2\sqrt a$ is also the minimal wave speed of traveling wave solutions of \eqref{Keller-Segel-eq} satisfying \eqref{Main-TW-eq}.

(2) Theorem \ref{Existence of traveling}  improves the results obtained in \cite{SaSh2}. 
Indeed,  assume that $2\chi\mu<b$. In \cite{SaSh2}, a positive constant $c^*(\chi, a, b, \lambda, \mu)>c^*_0$, which  depends on the parameters $\chi, a, b, \lambda,$ and  $\mu$, is obtained so that  for any $c>c^*(\chi, a, b, \lambda, \mu)$, \eqref{Keller-Segel-eq} has a traveling wave solution $(u(t,x),v(t,x))=(U(x-ct),V(x-ct))$ with speed $c$ connecting the constant solutions $(\frac{a}{b}, \frac{\mu}{\lambda}\frac{a}{b})$ and $(0,0)$.  It  left as an open question whether \eqref{Keller-Segel-eq} has a minimal wave speed $c_{\min}$ (i.e., whether there is $c_{\min}$ such that  \eqref{Keller-Segel-eq} has  a traveling wave solution $(u(t,x),v(t,x))=(U(x-ct),V(x-ct))$  connecting  $(\frac{a}{b}, \frac{\mu}{\lambda}\frac{a}{b})$ and $(0,0)$ with speed $c$ for any $c\ge c_{\min}$, and  has no such  traveling wave solution with speed $c<c_{\min}$). Note that when $\lambda\ge a$, $c^{**}=c_0^*$.
  Theorem 1.4 then implies  that  if $\lambda\ge a$, then  \eqref{Keller-Segel-eq} has a  minimal wave speed $c_{\min}$ and $c_{\min}=c_0^*$.
   Hence Theorem \ref{Existence of traveling} is an improvement of the results obtained in \cite{SaSh2}. 

\end{rk}

\subsection{Discussions}

In this subsection, we give some discussions on our main results.

Chemotaxis models describe the oriented movements of biological cells and organisms in response to chemical gradient. 
Consider \eqref{Keller-Segel-eq0} and its counterpart on a bounded domain $\Omega\subset\R^N$,
\begin{equation}\label{Keller-Segel-bounded-domain}
\begin{cases}
u_t=\Delta u-\nabla \cdot (\chi u \nabla v)+u(a-bu),\quad x\in\Omega\cr
0 =\Delta v-  \lambda v +\mu u,\quad x\in\Omega\cr
\frac{\p u}{\p n}=\frac{\p v}{\p n}=0,\quad x\in\Omega.
\end{cases}
\end{equation}
Suppose that $u(t,x)$ is the population density of certain biological cells and $v(t,x)$ is the  density of some chemical substance.
Then, in \eqref{Keller-Segel-eq0} (resp. \eqref{Keller-Segel-bounded-domain}), the term $\Delta u$ describes  the movement of the biological cells following random walk, which suggests that the cells move randomly from the places with higher cell density to the places with lower cell
density; the term $-\nabla \cdot (\chi u \nabla v)$ reflects the movement of the biological cells subject to the chemical substance, which suggests that the cells move from the places with lower chemical substance density to the places with higher chemical substance density when $\chi>0$;  the logistic term $u(a-bu)$ governs the local dynamics of the cell population; and the second equation in \eqref{Keller-Segel-eq0} (resp. \eqref{Keller-Segel-bounded-domain}) indicates that the chemical substance diffuses via random walk very quickly and is produced
over time by the biological cells.  
Both mathematically and biologically, it is important to investigate what dynamical scenarios  may be produced  by the interaction of these factors in  \eqref{Keller-Segel-eq0} (resp. \eqref{Keller-Segel-bounded-domain}), or how the chemotaxis affects the dynamics in \eqref{Keller-Segel-eq0} (resp. \eqref{Keller-Segel-bounded-domain}).

Observe that, in the absence of chemotaxis (i.e. $\chi=0$),  the dynamics of \eqref{Keller-Segel-eq0} is governed by  \eqref{fisher-kpp}, and the dynamics of
\eqref{Keller-Segel-bounded-domain} is governed by
\begin{equation}
\label{fisher-kpp-bounded-domain}
\begin{cases}
u_t=\Delta u+u(a-bu),\quad x\in\Omega\cr
\frac{\p u}{\p n}=0,\quad x\in\Omega.
\end{cases}
\end{equation}
It is known that the asymptotic dynamics of \eqref{fisher-kpp-bounded-domain} is completely determined by the logistic term $u(a-bu)$. More precisely, it is known that
$u\equiv \frac{a}{b}$ is the unique steady-state solution of \eqref{fisher-kpp-bounded-domain} ($\frac{a}{b}$ is
referred to as the carrying capacity of the system), and
for any given positive initial distribution $u_0(\cdot)\in C(\bar\Omega)$ ($u_0(x)\ge 0$ and $u_0(x)\not \equiv 0$),
the solution $u(t,x;u_0)$ of \eqref{fisher-kpp-bounded-domain} with $u(0,x;u_0)=u_0(x)$ converges to $\frac{a}{b}$
(i.e., the limiting distribution is $\frac{a}{b}$). The dynamics of \eqref{fisher-kpp} is recalled in the above. Among others, it is
known that if initially the population is inhabited in a bounded region, it spreads into the whole space at the speed $c_0^*=2\sqrt a$. Moreover, $c_0^*=2\sqrt a$ is the minimal wave speed of traveling wave solutions of \eqref{fisher-kpp} connecting $\frac{a}{b}$ and $0$.

Many authors have been studying possible dynamical scenarios induced from the chemotaxis in various chemotaxis models through the study of such models in bounded domains.
Very rich dynamics has been observed. For example,     when
 $a=b=0$, finite time blow-up  may occur in \eqref{Keller-Segel-bounded-domain} if either $N=2$ and the total initial population mass is large enough, or $N\geq 3$ (see \cite{BBTW,Hor,DirkandWinkler,win_arxiv}, and the references therein).
When  $a$ and $b$ are positive constants and $\lambda=\mu=1$,  if either $N\le 2$ or $b>\frac{N-2}{N}\chi$, then for any 
positive initial data $u_0\in C(\bar\Omega)$, \eqref{Keller-Segel-bounded-domain}   possesses a  unique bounded global classical solution $(u(x,t;u_0),v(x,t;u_0))$ with
 $u(x,0;u_0)=u_0(x)$, and hence the finite time blow-up phenomena in \eqref{Keller-Segel-bounded-domain}  is suppressed  to some  extent.
Moreover, if  $b>2\chi$, then $(\frac{a}{b},\frac{a}{b})$ is the unique positive steady-state solution
of  \eqref{Keller-Segel-bounded-domain} (with $\lambda=\mu=1$), and for any positive initial distribution $u_0\in C(\bar\Omega)$ ($u_0(x)\not \equiv 0$),
 $$
 \lim_{t\to\infty} \big[ \|u(\cdot;t;u_0)-\frac{a}{b}\|_{L^\infty(\Omega)}+\|v(\cdot,t;u_0)-\frac{a}{b}\|_{L^\infty(\Omega)}\big]=0
 $$
 (hence the chemotaxis does not affect the limiting distribution).
 But if $b<2\chi$, there may be more than one positive steady-state solutions of \eqref{Keller-Segel-bounded-domain} (see \cite{TeWi2}).

In the current paper, we investigate  the dynamics of chemotaxis models through the study of such models in unbounded domains,
in particular, through the study of chemotaxis models in  the whole space from the angle of spreading speed.
The first two authors of this paper have done a series of works in this direction and obtained several fundamental results.
For example,  as it is mentioned before, in the papers by \cite{SaSh1,SaSh2},   we  studied the dynamics of \eqref{Keller-Segel-eq0} from the angle of spreading speed,
and observed that if the population is initially inhabited in a bounded region, it spreads into the whole space as time evolves. Moreover, some 
explicit  lower bound $c_{\rm low}^{*}(\chi,a,b,\lambda,\mu)$ and upper bound $c_{\rm up}^*(\chi,a,b,\lambda,\mu)$ of the spreading speeds are obtained, and it is proved that both the lower and upper bounds converge to the spreading speed  $c_0^*=2\sqrt a$ of \eqref{fisher-kpp} as $\chi\to 0$.
But it  left as an open question whether or when  the solution of \eqref{Keller-Segel-eq0} with compactly supported initial function spreads at the speed $c_0^*=2\sqrt a$. 

The above open question is studied in the current paper and some satisfactory   answers are obtained. 
For example, assume $b>\chi\mu$ (i.e. the logistic damping constant $b$ is larger than the product of the chemotaxis sensitivity $\chi$ and
the production rate $\mu$ of the chemical substance produced by the biological cells).  It is proved in this paper that the spreading speed $c_0^*=2\sqrt a$ of \eqref{fisher-kpp} is always a lower bound of the spreading speeds of solutions of \eqref{Keller-Segel-eq0} with compactly supported or half space supported  initial distributions  (see Theorem \ref{spreading-speed-upper-bound} (2) and Theorem \ref{spreading-speed-upper-bound-of-front-like-initials}(2)), which implies that the chemotaxis does not slow down the spreading of population  with compactly supported or half space supported  initial distributions.
 If, in addition, $\lambda\ge a$ (i.e.  the degradation rate $\lambda$ of the  chemo-attractant is larger than or equal to the intrinsic growth
 rate $a$ of the biological cells),  it is proved that the spreading speed $c_0^*=2\sqrt a$ of \eqref{fisher-kpp} is also an upper bound
 of the spreading speeds of solutions of \eqref{Keller-Segel-eq0} with compactly supported or half space supported  initial distributions
 (see Theorem \ref{spreading-speed-upper-bound} (1) and Theorem \ref{spreading-speed-upper-bound-of-front-like-initials}(1)).
 Hence if $b>\chi\mu$ and $\lambda\ge a$, then the chemotaxis neither slows down nor speeds up the spreading of 
the cell population.
In general, we conjecture   that the  presence of the chemo-attractant does not increase the maximal spreading speed. 
 While our results do not settle completely the question of  the exact spreading speeds  of solutions to \eqref{Keller-Segel-eq0}, they provide a  satisfactory answer for some range of the parameters.  It would be of great mathematical interest to know whether the presence of the chemical really affects the spreading speed in general. We plan to devote some of our future works to address this question.
 
  In the paper by \cite{SaSh2}, we studied traveling wave solutions of \eqref{Keller-Segel-eq} connecting $(\frac{a}{b}, \frac{\mu}{\lambda}\frac{a}{b})$ and $(0,0)$, and found 
   a constant $c^*(\chi, a, b, \lambda, \mu)(>c^*_0=2\sqrt a)$ satisfying that,  for any $c>c^*(\chi, a, b, \lambda, \mu)$, \eqref{Keller-Segel-eq} has a traveling wave solution  with speed $c$ connecting the constant solutions $(\frac{a}{b}, \frac{\mu}{\lambda}\frac{a}{b})$ and $(0,0)$. Moreover, it is proved that $c^*(\chi, a, b, \lambda, \mu) \to c^*_0=2\sqrt a$ as $\chi\to 0$.
    It  left as an open question whether \eqref{Keller-Segel-eq} has a minimal wave speed $c_{\min}$  and if so, whether $c_{\min}=c_0^*$.  This open question is also studied in the current paper and some satisfactory  answers are obtained.
    For example, it is proved in this paper that, if $b>2\chi\mu$ and $\lambda\ge a$, then $c_0^*=2\sqrt a$ is the minimal wave speed
    of traveling wave solutions of \eqref{Keller-Segel-eq} connecting $(\frac{a}{b}, \frac{\mu}{\lambda}\frac{a}{b})$ and $(0,0)$.
    But when one of the conditions $b>2\chi\mu$ and $\lambda\ge a$ fails, the question of the existence of the minimum wave speed of \eqref{Keller-Segel-eq0} connecting the two constant equilibria remains an open problem.

 It should be pointed out  that the techniques developed in this work easily extend to the repulsion-chemotaxis models, that is $\chi<0$. Hence analogous results  can be obtained in such setting.

\medskip

The rest of the paper is organized as follows. In section 2, we present  some preliminary  results to be used in the proofs of our main results. In section 3, we study the spreading speed of solutions and prove Theorems \ref{spreading-speed-upper-bound}--\ref{Main-tm-2}. Finally in section 3, we study the existence and nonexistence of traveling wave solutions and prove Theorem \ref{Existence of traveling}.

\section{Preliminary lemmas}

In this section, we prove some lemmas to be used in the proofs of the main results in the later sections.

{For every $u\in C^b_{\rm unif}(\R)$, let
\begin{equation}\label{psi-definition}
\Psi(x;u)=\mu\int_{0}^{\infty}\int_{\R}\frac{e^{-\lambda s}e^{-\frac{|y-x|^2}{4s}}}{\sqrt{4\pi s}}u(y)dyds.
\end{equation}
It is well known that $\Psi(x;u)\in C^2_{\rm unif}(\R)$ and solves the elliptic equation
$$
\frac{d^2}{dx^2}\Psi(x;u)-\lambda\Psi(x;u)+\mu u=0.
$$

\begin{lem}
\label{new-lm1} It holds that
\begin{equation}\label{psi-definition-eq2}
\Psi(x;u)=\frac{\mu}{2\sqrt{\lambda}}\int_{\R}e^{-\sqrt{\lambda}|x-y|}u(y)dy
\end{equation}
and
\begin{equation}\label{space-derivative-of-psi-1}
\frac{d}{dx}\Psi(x;u)=-\frac{\mu}{2}e^{-\sqrt{\lambda }x}\int_{-\infty}^xe^{\sqrt{\lambda}y}u(y)dy + \frac{\mu}{2}e^{\sqrt{\lambda}x}\int_{x}^{\infty}e^{-\sqrt{\lambda}y}u(y)dy.
\end{equation}
\end{lem}

\begin{proof}
First, observe that the following identity holds.
\begin{equation}\label{nnew-eq1}
\int_{0}^{\infty}\frac{e^{-\frac{\beta^2}{ 4s}-s }}{\sqrt{4 \pi s}}ds=\frac{e^{-\beta}}{2}, \quad \forall \beta >0.
\end{equation}
 Indeed, note that by Residue Theorem $\frac{1}{2\pi}\int_{-\infty}^{\infty}\frac{e^{i\beta s}}{1+s^2}ds=\frac{e^{-\beta}}{2}$,  and
\begin{align*}
 \int_{0}^{\infty}e^{-(1+s^2)\tau} d\tau&=\frac{1}{1+s^2}.
 \end{align*}
 Hence,
\begin{align*}
 \frac{e^{-\beta}}{2}=\frac{1}{2\pi}\int_{-\infty}^{\infty}\int_{0}^{\infty}e^{-(1+s^2)\tau} e^{i\beta s}d\tau ds&=\frac{1}{2\pi}\int_{0}^{\infty}e^{-\tau}\int_{-\infty}^{\infty}e^{-\tau(s^2-\frac{i\beta}{\tau}s)}dsd\tau\cr
&=\frac{1}{2\pi}\int_{0}^{\infty}e^{-\tau}e^{-\frac{\beta ^2}{4\tau}}(\int_{-\infty}^{\infty}e^{-\tau(s-\frac{i\beta}{2\tau})^2}ds)d\tau\cr
&=\int_{0}^{\infty}\frac{e^{-\frac{\beta^2}{ 4\tau}-\tau }}{\sqrt{4 \pi \tau}}d\tau.
 \end{align*}

Next using Fubini's Theorem, one can exchange the order of integration  in \eqref{psi-definition} to obtain
\begin{align}\label{nnew-eq2}
\Psi(x;u)=\mu \int_{0}^{\infty}\int_{\R}\frac{e^{-\lambda s}e^{-\frac{|x-y|^2}{4s}}}{\left[4\pi s\right]^{\frac{1}{2}}}u(y)dyds=\mu \int_{\R}\left[\int_{0}^{\infty}\frac{e^{-\frac{|x-y|^2}{4s}-\lambda s}}{\sqrt{4\pi s}}ds\right]u(y)dy.
\end{align}
By the change of variable $\tau=\lambda s$ and taking $\beta =\sqrt{\lambda}|x-y|$, it follows from \eqref{nnew-eq1} that
$$
\int_{0}^{\infty}\frac{e^{-\frac{|x-y|^2}{4s}-\lambda s}}{\sqrt{4\pi s}}ds=\frac{1}{\sqrt{\lambda}}\int_{0}^{\infty}\frac{e^{-\frac{\beta^2}{ 4\tau}-\tau }}{\sqrt{4 \pi \tau}}d{\tau} =\frac{1 }{2\sqrt{\lambda}}e^{-\sqrt{\lambda}|x-y|}.
$$
This together with   \eqref{nnew-eq2} implies that
$$
\Psi(x;u)=\mu \int_{0}^{\infty}\int_{\R}\frac{e^{-\lambda s}e^{-\frac{|x-y|^2}{4s}}}{\left[4\pi s\right]^{\frac{1}{2}}}u(y)dyds=\frac{\mu}{2\sqrt{\lambda}}\int_{\R}e^{-\sqrt{\lambda}|x-y|}u(y)dy.
$$
Thus \eqref{psi-definition-eq2} holds.

Now, by  \eqref{psi-definition-eq2},
$$
\Psi(x;u)=\frac{\mu}{2\sqrt{\lambda}}\int_{-\infty}^xe^{-\sqrt{\lambda}(x-y)}u(y)dy+\frac{\mu}{2\sqrt{\lambda}}\int_{x}^\infty e^{-\sqrt{\lambda}(y-x)}u(y)dy.
$$
\eqref{space-derivative-of-psi-1} then follows from a direction calculation.
\end{proof}

\begin{lem}\label{lem-001}
For every $u\in C^b_{\rm unif}(\R)$, $u(x)\ge 0$, it holds that
\begin{equation}\label{estimate-on-space-derivative-1}
\left| \frac{d}{dx}\Psi(x;u)\right|\leq \sqrt{\lambda}\Psi(x;u),\ \quad  \forall\ x\in\R.
\end{equation}
In particular, for every solution $(u(t,x;u_0),v(t,x;u_0))$ of \eqref{Keller-Segel-eq} with $u_0\geq 0$ it holds that
\begin{equation*}
|v_{x}(t,x;u_0)|\leq \sqrt{\lambda}v(t,x;u_0),\quad \forall\ x\in\R, \ \forall\ t>0.
\end{equation*}
Furthermore,  if
\begin{equation}
\frac{(\kappa-\sqrt{\lambda})_{+}}{(\kappa +\sqrt{\lambda})}\leq \frac{2(b-\chi\mu)}{\chi\mu},
\end{equation} it holds that
\begin{equation}\label{estimate-on-space-derivative-2}
\chi\kappa\Psi_x(t,x;u)-\chi\lambda\Psi(t,x;u)-(b-\chi\mu)Me^{-\kappa x}\leq 0, \quad \forall x\in\R,
\end{equation}
whenever $0\leq u(x)\leq Me^{-\kappa x}$ for some positive real numbers $\kappa>0$ and $M>0$.
\end{lem}

 \begin{proof}
 First, by \eqref{psi-definition-eq2} and \eqref{space-derivative-of-psi-1}, we have
 \begin{align*}
|\frac{d}{dx}\Psi(x;u)|&=|-\frac{\mu}{2}e^{-\sqrt{\lambda }x}\int_{-\infty}^xe^{\sqrt{\lambda}y}u(y)dy + \frac{\mu}{2}e^{\sqrt{\lambda}x}\int_{x}^{\infty}e^{-\sqrt{\lambda}y}u(y)dy|\\
&\le \frac{\mu}{2}\int_{\R} e^{-\sqrt{\lambda}|y-x|}u(y)dy\\
&=\sqrt \lambda \Psi(x;u).
\end{align*}
 This implies \eqref{estimate-on-space-derivative-1}.

Next, we prove \eqref{estimate-on-space-derivative-2}. It follows from \eqref{psi-definition} and \eqref{space-derivative-of-psi-1} that
 \begin{align*}
&\chi\kappa\Psi_x(t,x;u)-\chi\lambda\Psi(t,x;u)\cr
=&- \frac{\chi \mu}{2}(\kappa+\sqrt{\lambda})e^{-\sqrt{\lambda }x}\int_{-\infty}^xe^{\sqrt{\lambda}y}u(y)dy    -\frac{\chi \mu}{2}(\sqrt{\lambda}-\kappa)e^{\sqrt{\lambda}x}\int_x^{\infty}e^{-\sqrt{\lambda}y}u(y)dy\cr
\leq &  \frac{\chi \mu}{2}(\kappa-\sqrt{\lambda})_{+}e^{\sqrt{\lambda}x}\int_x^{\infty}e^{-\sqrt{\lambda}y}u(y)dy\cr
\leq& \frac{\chi \mu M}{2}(\kappa-\sqrt{\lambda})_{+}e^{\sqrt{\lambda}x}\int_x^{\infty}e^{-\sqrt{\lambda}y}e^{-\kappa y}dy\cr
=& \frac{\chi \mu M}{2(\kappa +\sqrt{\lambda})}(\kappa-\sqrt{\lambda})_{+}e^{-\kappa x}.
 \end{align*}
 Hence, \eqref{estimate-on-space-derivative-2} follows.
\end{proof}

\begin{lem}
\label{lem-001-1}
 Assume that $b>\chi\mu$. Let $c_{\kappa}=\frac{\kappa^2+a}{\kappa}$ with $ 0<\kappa\leq \sqrt{a}$ satisfying
$$\frac{(\kappa-\sqrt{\lambda})_{+}}{(\kappa +\sqrt{\lambda})}\leq \frac{2(b-\chi\mu)}{\chi\mu}.$$ The following hold.
\begin{enumerate}
\item[(i)]
 For any $u_0\ge 0$ with  nonempty compact support
 and  any $M\gg \frac{a}{b-\chi\mu}$  satisfying
\begin{equation}\label{u_0-cond1}
\max\{u_0(x),u_0(-x)\}\leq U^+(x):=\min\{M,Me^{-\kappa x}\},\quad \forall\ x\in\R,
\end{equation}
there holds
\begin{equation}
\label{new-eq1}
u(t,x;u_0)\leq M e^{-\kappa(|x|-c_{\kappa}t)}, \quad\forall\ x\in\R,\ t\geq 0.
\end{equation}
\item[(ii)]
 For any $u_0\in C^b_{\rm unif}(\R)$, $u_0(x)>0 $,   and  any  $M\gg \frac{a}{b-\chi\mu}$  satisfying
\begin{equation}\label{u_0-cond2}
u_0(x)\leq U^+(x):=\min\{M,Me^{-\kappa x}\},\quad \forall\ x\in\R,
\end{equation}
there holds
\begin{equation}
\label{new-eq1-2}
u(t,x;u_0)\leq M e^{-\kappa(x-c_{\kappa}t)}, \quad\forall\ x\in\R,\ t\geq 0.
\end{equation}

\end{enumerate}
\end{lem}
}

\begin{proof} We shall only prove $(i)$ since $(ii)$ can be proved by the similar arguments.

For any given  $T>0$, let
$$
\mathcal{E}^T:=\{u\in C([0,T], C^b_{\rm unif}(\R))\, | \, u(0,\cdot)=u_0(\cdot)\ \text{and}\ 0\leq u(t,x)\leq {U}^{+}(x)\ \forall\ x\in\R, \ 0\leq t\leq T\}.
$$
For every $u\in\mathcal{E}^T$, let $\Phi(t,x;u)$ denote the solution of
\begin{equation*}
\begin{cases}
\Phi_t=\mathcal{A}_u(\Phi), \ x\in\R, 0<t\leq T,\cr
\Phi(0,x)=u_0(x),\quad x\in\R.
\end{cases}
\end{equation*}
where $$\mathcal{A}_u(\Phi):=\Phi_{xx}+(c_\kappa-\chi\Psi_x(\cdot,\cdot;u))\Phi_x+(a-\chi\lambda\Psi(\cdot,\cdot;u)-(b-\chi\mu)\Phi)\Phi$$
and $\Psi(x,t;u)$ is the solution of
$$
\Psi_{xx}-\lambda \Psi+\mu u=0,\quad x\in\R.
$$

Observe that
\begin{equation}\label{w-eq0}
\mathcal{A}_u(Me^{-\kappa x})=\left(\chi\kappa\Psi_x(t,x;u)-\chi\lambda\Psi(t,x;u)-(b-\chi\mu)Me^{-\kappa x}\right)Me^{-\kappa x}.
\end{equation}
It follows from Lemma \ref{lem-001} and \eqref{w-eq0} that
$$
\mathcal{A}_u(Me^{-\kappa x})\leq 0.
$$
Thus, since $u_0(x)\leq Me^{-\kappa x}$ for every $x\in\R$, by comparison principle for parabolic equations, we conclude that
$$
\Phi(t,x;u)\leq Me^{-\kappa x}, \quad \forall\ x\in\R, \forall\ t\geq 0.
$$
On the other hand, since $M\geq \frac{a}{b-\chi\mu}$, we have that
$$
\mathcal{A}_u(M)=(a-\chi\lambda \Psi(t,x;u)-(b-\chi\mu)M)M\leq 0.
$$
Thus, since $u_0(x)\leq M$ for every $x\in\R$, by comparison principle for parabolic equations again, we conclude that
$$
\Phi(t,x;u)\leq M, \quad \forall\ x\in\R, \forall\ t\geq 0.
$$
Therefore, we have that
$$
0\leq \Phi(t,x;u)\leq U^+(x),\quad\forall\ x\in\R,\ 0\leq t\leq T,\ \ u\in\mathcal{E}^T.
$$

Following the arguments of the proof of \cite[Theorem 3.1 ]{SaSh2}, it can be shown that the function  $\Phi \ :\mathcal{E}^T\ni u \mapsto \Phi(\cdot,\cdot,u)\in \mathcal{E}^T $ is continuous and compact in the open compact topology. Hence by Schauder's fixed point theorem there is $u^*\in \mathcal{E}^T$, such that $\Phi(u^*)=u^*$. Note that $(u^*(t,x- c_{\kappa}t),v^*(t,x- c_{\kappa}t))$ is also a solution of \eqref{Keller-Segel-eq}, with $u^*(0,x)=u_0(x)$ for every $x\in\R$. Hence, by uniqueness of the solution to \eqref{Keller-Segel-eq}, we conclude that
$$
u^*(t,x-c_\kappa t)=u(t,x;u_0),\quad \forall\ x\in\R, \ 0\leq t\leq T.
$$
Hence $u(t,x;u_0)\in \mathcal{E}^T$. Since $T$ was arbitrary chosen, we obtain that
$$
u(t,x;u_0)\leq U^{+}(x-c_{\kappa}t)\leq M e^{-\kappa(x-c_{\kappa}t)}, \quad\forall\ x\in\R,\ t\geq 0.
$$

Similar arguments as in the above yield that
$$
u(t,-x;u_0)\leq U^{+}(x-c_{\kappa}t)\leq M e^{-\kappa(x-c_{\kappa}t)}, \quad\forall\ x\in\R,\ t\geq 0.
$$
Thus
$$
u(t,x;u_0)\leq M e^{-\kappa(|x|-c_{\kappa}t)}, \quad\forall\ x\in\R,\ t\geq 0.
$$
The lemma is proved.
\end{proof}

For every $0<\kappa<\tilde{\kappa}<\min\{\sqrt{a},\sqrt{\lambda}\}$ with $\tilde{\kappa}<2\kappa$ and $D\ge 1$, consider the functions $\varphi_{\kappa}(x)$, $\overline{U}_{\kappa,D}(x)$, and $\underline{U}_{\kappa,D} (x)$  given by  $$
\varphi_{\kappa}(x)=e^{-\kappa x},
$$
\begin{equation}\label{l-003}
\overline{U}_{\kappa,D}(x):=\min\{D, \varphi_{\kappa}(x)+D\varphi_{\tilde{\kappa}}(x)\},
\end{equation}
and
\begin{equation}\label{l-004}
\underline{U}_{\kappa,D}(x)=\begin{cases}
\varphi_{\kappa}(x)-D\varphi_{\tilde{\kappa}}(x), \quad x\geq x_{\kappa,D}\cr
\varphi_{\kappa}(x_{\kappa,D})-D\varphi_{\tilde{\kappa}}(x_{\kappa,D}), \quad x\leq x_{\kappa,D},
\end{cases}
\end{equation}
where $x_{\kappa,D}$ satisfies
\begin{equation}\label{l-005}
\max\{\varphi_{\kappa}(x)-D\varphi_{\tilde{\kappa}}(x)\, |\, x\in\R\}=\varphi_{\kappa}(x_{\kappa,D})-D\varphi_{\tilde{\kappa}}(x_{\kappa,D}).
\end{equation}

\begin{lem}
For every $0\leq u(x)\leq \overline{U}_{\kappa,D}(x)$, it holds that
\begin{equation}\label{l-001}
\Psi(x;u)\leq \frac{\mu}{\lambda-\kappa^2}\varphi_{\kappa}(x)+\frac{D\mu}{\lambda-\tilde{\kappa}^2}\varphi_{\tilde{\kappa}}(x)
\end{equation}
and
\begin{equation}\label{l-002}
|\Psi_{x}(x;u)|\leq \mu\left(\frac{1}{\sqrt{\lambda-\kappa^2}}+\frac{\kappa}{\lambda-\kappa^2}\right)\varphi_{\kappa}(x)+D\mu\left(\frac{1}
{\sqrt{\lambda-\tilde{\kappa}^2}}+\frac{\tilde{\kappa}}{\lambda-\tilde{\kappa}^2}\right)\varphi_{\tilde{\kappa}}(x).
\end{equation}
\end{lem}

\begin{proof}It follows from proper modification of the proof of \cite[Lemma 2.2]{SaSh2}.
\end{proof}


\begin{lem}
\label{new-lm2}
Assume $b>\chi\mu$.
For every $R\gg 1$, there are $C_R\gg 1$ and $\varepsilon_R>0$ such that for any $u_0\in C_{\rm unif}^b(\R)$ with $u_0\ge 0$, any
$x_0\in \R$, and any $t\ge 0$,  we have
 \begin{equation}\label{nnew-eq3}
|\chi v_x(t,\cdot;u_0)|_{L^{\infty}(B_{\frac{R}{2}}(x_0))} + |\chi\lambda v(t,\cdot;u_0)|_{L^{\infty}(B_{\frac{R}{2}}(x_0))}\leq C_{R}\|u(t,\cdot;u_0)\|_{L^{\infty}(B_{R}(x_0))}+\varepsilon_R M
\end{equation}
with $\lim_{R\to\infty}\varepsilon_R=0$, where $M:=\max\{\|u_0\|_{\infty},\frac{a}{b-\chi\mu}\}$.
\end{lem}

\begin{proof} We first note that
$$
\|u(t,\cdot;u_0)\|_{\infty}\leq \max\{\|u_0\|_{\infty},\frac{a}{b-\chi\mu}\}=M,\quad \forall\ t\geq 0.
$$
Observe from  Lemma \ref{new-lm1} that
$$
v(t,x;u_0)=\frac{\mu}{2\sqrt{\lambda}}\int_{\R}e^{-\sqrt{\lambda}|z|}u(t,x-z;u_0)dz.
$$
Hence
$$
|v(t,x;u_0)|\le \frac{\mu}{2\sqrt{\lambda}}\int_{\mathbb{B}_{\frac{R}{2}}}e^{-\sqrt{\lambda}|z|}u(t,x-z;u_0)dz + \frac{\mu}{2\sqrt{\lambda}}\Big[\int_{\R\setminus\mathbb{B}_{\frac{R}{2}}}e^{-\sqrt{\lambda}|z|}dz\Big]\|u(t,\cdot;u_0)\|_{\infty}.
$$
Thus, since $x\in{B}_{\frac{R}{2}}(x_0)$ and $z\in\mathbb{B}_{\frac{R}{2}}$ imply that $x-z\in\mathbb{B}_{\R}(x_0)$, we obtain that
$$
\|v(t,\cdot;u_0)\|_{ {L^{\infty}({B}_{\frac{R}{2}}(x_0))}}\le \Big[\frac{\mu}{2\sqrt{\lambda}}\int_{\mathbb{B}_{\frac{R}{2}}}e^{-\sqrt{\lambda}|z|}dz\Big]\|u(t,\cdot;u_0)\|_{ {L^{\infty}({B}_{R}(x_0)) }}+ \frac{\mu}{2\sqrt{\lambda}}\Big[\int_{\R\setminus\mathbb{B}_{\frac{R}{2}}}e^{-\sqrt{\lambda}|z|}dz\Big]M.
$$
This combined with \eqref{estimate-on-space-derivative-1} yields \eqref{nnew-eq3}.
\end{proof}

\begin{lem}
\label{new-lm3}
{ Assume $b>\chi\mu$.} For every $p>1$, $t_0>0$, $s_0\geq 0$, $R>0$, and  $u_0\in C_{\rm unif}^b(\R)$, taking $M:=\max\{\|u_0\|_{\infty},\frac{a}{b-\chi\mu}\}$, there is $C_{t_0,s_0,R,M,p}$ such that
\begin{equation}\label{nnew-eq4}
u(t,x;u_0)\le C_{t_0,s_0,R,M,p}[u(t+s,y;u_0)]^{\frac{1}{p}}(M+1),\quad \forall\ s\in[0,s_0], \ t\geq t_0,  \ |x-y|\leq R.
\end{equation}
 \end{lem}

\begin{proof}It can be proved by the arguments of \cite[Lemma 2.2]{FhCh}.
\end{proof}

By \eqref{nnew-eq3} and \eqref{nnew-eq4} with $p>1$, $s_0=0$ and $t_0=1$, we have
\begin{align}
\label{E2-1}
|\chi v_x(t,x;u_0)|+\chi \lambda v(t,x;u_0)&\le C_{R,p} \big(u(t,x;u_0)\big)^{\frac{1}{p}}+\varepsilon_R M\nonumber\\
&= C_{R,p} \big(u(t,x;u_0)\big)^{\frac{1}{p}}+\varepsilon_R M \quad \forall\,\, t\ge 1,\,\, x\in\R,
\end{align}
where  $C_{R,p}=C_R\cdot  C_{1,0,R,M,p}\cdot (M+1) (>0)$.

\section{Spreading speeds}

In this section we derive an explicit upper bound on the spreading speeds of solutions of \eqref{Keller-Segel-eq} with  nonempty compactly  supported initial functions or  exponentially decay initial functions, and
 prove Theorems \ref{spreading-speed-upper-bound}  and  \ref{Main-tm-2}.

\begin{proof}[Proof of Theorem \ref{spreading-speed-upper-bound}]
(1) First, note that for any $c>c^*$, there is $0<\kappa<\sqrt a$ such that
$$\frac{(\kappa-\sqrt{\lambda})_{+}}{(\kappa +\sqrt{\lambda})}\leq \frac{2(b-\chi\mu)}{\chi\mu}
$$
and
$$
c>c_\kappa=\frac{\kappa^2 +a}{\kappa}.
$$
Then
by Lemma \ref{lem-001-1},  we have
$$
\lim_{t\to\infty} \sup_{|x|\ge ct}u(t,x;u_0)=0.
$$
This implies that
$$
c_+^*(\chi,a,b,\lambda,\mu)\le c^*.
$$
 Note that {\bf (H)} implies $c^*=2\sqrt a$.   Thus Theorem \ref{spreading-speed-upper-bound} (1) follows.

\smallskip

(2) Let $0<c<2\sqrt{a}$ be given and set $M=1+a+\max\{\|u_0\|_{\infty},\frac{a}{b-\chi\mu}\}$. By \eqref{E2-1}, it follows that, for any  $R\gg 1$, $(u(t,x;u_0),v(t,x;u_0))$ satisfies
\begin{equation}\label{AA-eq1}
u_t\geq u_{xx}-\chi v_xu_x+u(a-\varepsilon_R M- C_{R,p} u^{\frac{1}{p}}-(b-\chi\mu)u), \quad t\geq 1,\ x\in\R.
\end{equation}
Let $p=2$. Choose  $R\gg 1$ and $0<\eta\ll \min\{1,a\}$ such that $
\varepsilon_{R}M<\frac{\eta}{4}$,
\begin{equation}\label{AA-eq2}
(c+\eta)^2<4(a-\varepsilon_R M),
\end{equation}
and
\begin{equation}\label{AA-eq2-1}
|\chi v_x|\leq C_{R}\sqrt {u(t,x;u_0)}+\frac{\eta}{4}, \quad t\geq 1,\ x\in\R.
\end{equation}
Define
$$
A(t,x)=\frac{\chi v_x}{\max\{1, |\chi v_x|\eta^{-1}\}}, \quad t\geq 1,\ x\in\R.
$$
From this point, the remaining part of the proof is completed in four steps.

\smallskip

\noindent {\bf Step 1.} In this step we construct some sub-solution for \eqref{AA-eq1}.

\smallskip

First, chose $0<\eta_1\ll 1$  satisfying
 \begin{equation}\label{AA-eq2-2}
 C_R\sqrt{\eta_1}+\eta_1<\frac{\eta}{4}.
 \end{equation}
 Let  $\kappa=\min\{\sqrt{a},\sqrt{\lambda}\}$ and $c_{\kappa}=\frac{\kappa^2+a}{\kappa}$. By Lemma \ref{lem-001-1}, we have
\begin{equation}\label{AA-eq3}
u(t,x;u_0)\leq Me^{-\kappa(|x|-c_{\kappa}t)}, \quad\forall\ t\ge 1, x\in\R.
\end{equation}
Choose $m_0\gg 1$   such that
\begin{equation}\label{AA-eq3-1}
Me^{-\kappa m_0}<\eta_1.
\end{equation}
For $N>c_{\kappa}+m_0+1$ (fixed), set
$$
m_1:=\frac{1}{3}\min\{u(t,x)\, |\,  1\leq t\leq N+1, |x|\leq  (c_{\kappa}+N)N+m_0+1\}.
 $$

Next,  let $\underline{u}_1\in C^b_{\rm unif}(\R)\setminus\{0\}$ be  such that
$$
0\leq \underline{u}_1(x)\leq m_1\quad \text{and}\quad  \underline{u}_1(x)=0 \ \forall\ |x|\geq2.
$$
Let $\underline{u}(t,x)$ be the solution of
\begin{equation}\label{AA-eq4}
\begin{cases}
\underline{u}_t=\underline{u}_{xx}-A(t,x)\underline{u}_x+\underline{u}(a-\varepsilon_RM -C_{R,P} \underline{u}^{\frac{1}{p}}-\frac{4(M+a)}{\eta_1}\underline{u}), t>1, x\in\R\cr
\underline{u}(1,x)=\frac{\eta_1 }{4M+ 4\eta_1 e^{(a-\varepsilon_RM)N}}\underline{u}_1(x),\quad x\in\R.
\end{cases}
\end{equation}
Clearly, $u(t,x)\equiv \eta_1$ is a super-solution of \eqref{AA-eq4} and $\|\underline{u}(1,\cdot)\|_{\infty}<\eta_1$. Thus, by comparison principle for parabolic equations that
$$
\underline{u}(t,x)<\eta_1, \quad\forall\ t\geq 1, \ x\in\R.
$$
Furthermore, since $\|A(t,\cdot)\|_{\infty}\leq \eta$ for every $t\geq 1$, then by \eqref{AA-eq2} it holds that
$$
2\sqrt{a-\varepsilon_RM}-\sup_{t\ge 1}\|A(t,\cdot)\|_{\infty}\geq 2\sqrt{a-\varepsilon_RM}-\eta>c.
$$
It then follows from \cite[Theorem 1.2]{BHN} that
\begin{equation}\label{AA-eq5}
\liminf_{t\to\infty}\inf_{|x|\le ct}\underline{u}(t,x)>0.
\end{equation}

\noindent {\bf Step 2.} In this step we compare $\underline{u}(t,x)$ and $u(t,x;u_0)$ and show that
\begin{equation}\label{AA-eq6}
\underline{u}(t,x)< u(t,x;u_0), \quad \forall\ |x|\le (c_{\kappa}+N)t +m_0 , \ t\ge 1.
\end{equation}

Suppose, by contradiction that \eqref{AA-eq6} does not hold. Then there is $t_{\rm inf}\in[1,\infty)$ satisfying
$$
t_{\rm inf}:=\inf\{t\in(1,\infty)\, |\,  \exists  x_t\in\R, {\rm satisfying }\quad  \underline{u}(t,x_t)\geq u(t,x_t), |x_t|\leq (c_{\kappa}+N)t+m_0 \}.
$$
Note that
$$
 \|\underline{u}(t,\cdot)\|_{\infty}\le \frac{\eta_1 m_1}{4M+4\eta_1 e^{(a-\varepsilon_RM)N}}e^{(a-\varepsilon_RM)t}< m_1<u(t,x), \quad \forall\ 1\leq t\leq N , |x|\le (c_{\kappa}+N)N+m_0.
$$
Hence
$$t_{\rm inf}\geq N. $$
Moreover, there is $x_{\rm inf}\in\R$ such that $|x_{\rm inf}|\leq (c_{\kappa}+N)t_{\rm inf}+m_0$,
\begin{equation}\label{AA-eq7}
\eta_1>\underline{u}(t_{\rm inf},x_{\rm inf})=u(t_{\rm inf},x_{\rm inf}),
\end{equation}
and
$$
\underline{u}(t,x)<u(t,x),\quad\ |x|\leq (c_{\kappa}+N)t+m_0, \quad 1\le t<t_{\rm inf}.
$$
We have the following two cases.

\smallskip

\noindent {\bf Case 1.} $|x_{\rm inf}|< (c_{\kappa}+N)t_{\rm inf}+m_0$.

In this case, by inequalities \eqref{AA-eq2-1}, \eqref{AA-eq2-2}, and \eqref{AA-eq7}, there is $0<\delta\ll 1$ such that  $[t_{\rm inf}-\delta,t_{\rm inf}]\times [x_{\rm inf}-\delta,x_{\rm inf}+\delta]\subset \{(t,y)\,:\,|y|<(c_{\kappa}+N)t+m_0 \}$ and
$$
A(t,x)=\chi v_x(t,x;u_0)
,\quad \forall\ t_{\rm inf}-\delta\le t\leq t_{\rm inf},\quad x_{\rm inf}-\delta\leq x\leq x_{\rm inf}+\delta.$$
Note that
$$
\underline{u}(t_{\rm inf}-\delta,x)<u(t_{\rm inf}-\delta,x) \quad \forall x\in[x_{\rm inf}-\delta,x_{\rm inf}+\delta]
$$
and
$$
\underline{u}(t,x_{\rm inf}\pm\delta)\leq u(t,x_{\rm inf}\pm\delta) \quad \forall t_{\rm inf}-\delta\leq t\leq t_{\rm inf}.
$$
Thus, by the comparison principle for parabolic equations, we have
$$
\underline{u}(t,x)<u(t,x) \quad\forall t_{\rm inf}-\delta\leq t\leq t_{\rm inf}, \quad x_{\rm inf}-\delta\leq x\leq x_{\rm inf}+\delta.
$$
In particular,
$$
\underline{u}(t_{\rm inf},x_{\rm inf})<u(t_{\rm inf},x_{\rm inf}).
$$
Which contradicts to \eqref{AA-eq7}.

\smallskip

\noindent {\bf Case 2.} $|x_{\rm inf}|=(c_{\kappa}+N)t_{\rm inf}+m_0$.

In this case, without loss of generality, we may suppose that $x_{\rm inf}=(c_{\kappa}+N)t_{\rm inf}+m_0$.  Let $0<\delta <N-1$ be fixed. Observe that for every  $t\in[t_{\rm inf}-\delta,t_{\rm inf})$ and $x\in[x_{\rm inf}-\delta,\infty)$
$$
|x|-c_{\kappa}t\geq x_{\rm inf}-\delta-c_{\kappa}t_{\rm inf}= N t_{\rm inf}-\delta+m_0 >m_0.
$$
Thus, by \eqref{AA-eq3}, \eqref{AA-eq2-2}, and \eqref{AA-eq2-1} we obtain that
\begin{equation*}
A(t,x)=\chi v_x(t,x;u_0)
,\quad \forall\ t_{\rm inf}-\delta\le t\leq t_{\rm inf},\quad x\geq x_{\rm \inf}-\delta.
\end{equation*}
Whence, since
$$
\underline{u}(t,x_{\rm inf}-\delta)\leq u(t,x_{\rm inf}-\delta), \quad \forall\ t_{\rm inf}-\frac{\delta}{c_{\kappa}+N} \leq t\leq t_{\rm inf},
$$
in order to conclude that $ \underline{u}(t_{\rm inf},x_{\rm inf})<u(t_{\rm inf},x_{\rm inf})$ and obtain a contradiction as in the previous case, it is enough to show that
\begin{equation}\label{AA-eq8}
\underline{u}(t_{\rm inf}-\frac{\delta}{c_{\kappa}+N},x)< u(t_{\rm \inf}-\frac{\delta}{c_{\kappa}+N},x),\quad \forall\ x\geq x_{\rm inf}-\delta.
\end{equation}
 So, to complete the proof of this step it remains to prove \eqref{AA-eq8}.  Observe that
$$
A(t,x)=\chi v_x(t,x;u_0), \quad \forall\ x\geq (c_{\kappa}+N)t+m_0, t\geq 1,
$$
$$
\underline{u}(1,x)< u(1,x),\quad \forall\ x\in\R,
$$
and
$$
\underline{u}(t,(c_{\kappa}+N)t+m_0)< u(t,(c_{\kappa}+N)t+m_0), \quad \forall 1\leq t\leq t_{\rm inf}-\frac{\delta}{c_{\kappa}+N}.
$$
Thus by comparison principle for parabolic equations, we conclude that \eqref{AA-eq8} holds.

\smallskip

\noindent {\bf Step 3}. We conclude the proof of \eqref{T1-eq1} here.

\smallskip

By  \eqref{AA-eq5} and \eqref{AA-eq6}, we deduce that
$$
\liminf_{t\to\infty}\inf_{|x|\le ct}u(t,x)>0.
$$
Which completes the proof of \eqref{T1-eq1}.

\smallskip

\noindent {\bf Step 4.} In this step, we prove \eqref{T1-eq2}.

\smallskip

Suppose that $2\chi\mu<b$ and suppose by contradiction that \eqref{T1-eq2} does not hold. Then there exist $0<c<2\sqrt{a}$,  $t_n\to\infty$ and $|x_n|\leq ct_n$ for every $n\geq 1$  such that
 \begin{equation}\label{xx-q}
\inf_{n\geq 1}|u(t_n,x_n)-\frac{a}{b}|>0.
 \end{equation}
 Consider the sequence $(u^n(t,x),v^n(t,x))=(u(t+t_n,x+x_n;u_0),v(t+t_n,x+x_n))$, using estimates for parabolic equations, without loss of generality we may suppose that $(u^n(t,x),v^n(t,x))\to (u^*(t,x),v^*(t,x))$ locally uniformly in $C^{1,2}(\R\times\R)$. Furthermore, $(u^*(t,x),v^*(t,x))$  is an entire solution of \eqref{Keller-Segel-eq}. But, it holds that
 $$
u^*(t,x)\geq \liminf_{ {\tau\to\infty}}{ \inf_{|y|\leq(c+\frac{2\sqrt{a}-c}{2})\tau}}u(\tau,y;u_0)>0.
 $$
 Thus, since $2\chi\mu<b$, by the stability of the positive constant equilibrium $(\frac{a}{b},\frac{a\mu}{b\lambda})$, we must have $u^*(t,x)=\frac{a}{b}$ for every $t,x\in\R$. In particular, $u^*(0,0)=\frac{a}{b}$, which contradicts to \eqref{xx-q}.
\end{proof}

\begin{proof} [Proof of Theorem \ref{spreading-speed-upper-bound-of-front-like-initials}]
It can be proved by the similar arguments as those in Theorem \ref{spreading-speed-upper-bound}.
\end{proof}

 To prove  Theorem \ref{Main-tm-2} we first recall the following result established in \cite{SaSh6_1}.

\begin{tm}\cite[Theorem 1.2 (i)]{SaSh6_1}
 \label{persitence-tm}
 Assume $\chi\mu<b$. For every  $ \delta >0$ and $M>0$  there is $0<\underline{m}( \delta,\chi,\mu,a,b,\lambda,M)<\overline{m}( \delta,\chi,\mu,a,b,\lambda,M)<\infty$ such that for every $u_0\in C^{b}_{\rm unif}(\R)$ satisfying
$$
\delta\leq u_0(x)\leq M \quad  \forall x\in\R,
$$
then
$$
\underline{m}(\delta,\chi,\mu,a,b,\lambda,M)\leq u(t,x;u_0)\leq \overline{m}(\delta,\chi,\mu,a,b,\lambda,M),\quad \forall\ x\in\R, t\geq 0.
$$
\end{tm}

\medskip

Next, we present the proof of Theorem \ref{Main-tm-2}.

\medskip

\begin{proof}[Proof of Theorem \ref{Main-tm-2}(1)]

 Let $u_0\in C^b_{\rm unif}(\R)$ satisfy \eqref{u-0-condition-1}. Then there is $M\gg 1$ such that  $u_0(x)\leq\min\{M, M e^{-\kappa x}\}$ for every $x\in\R$. Therefore, \eqref{Spreadind-speed-eq1} follows from \eqref{new-eq1-2}. So, it remains to prove that \eqref{Spreadind-speed-eq1-1} holds.

Let $m=\frac{1}{2}\inf_{x\leq 0}u_0(x)$ and $M\gg 1$ be chosen as above. Let $0<\underline{m}(m,\chi,\mu,a,b,\lambda,M)<\overline{m}(m,\chi,\mu,a,b,M)$ be given by Theorem \ref{persitence-tm}. Then for there every $T>0$ it holds that
\begin{equation}\label{Z-eq0}
\liminf_{x\to-\infty}u(T,x;u_0)\geq \underline{m}(m,\chi,\mu,a,b,\lambda,M).
\end{equation}

By Lemmas \ref{new-lm2}  and \ref{new-lm3}, it follows that $(u(t,x;u_0),v(t,x;u_0))$ satisfies
\begin{equation}
u_t\geq u_{xx}-\chi v_xu_x+u(a-\varepsilon_R M - C_{R,P} u^{\frac{1}{p}}-(b-\chi\mu)u),\quad t\ge 1, \ x\in\R.
\end{equation}
Let $0<\varepsilon\ll 1$ be fixed.  Choose $R\gg 1$ such that $0<\varepsilon_RM<\frac{\varepsilon}{4}$ and   $\kappa^2<a-\varepsilon_R M$.  For every $0<\delta\ll a-\varepsilon_R M$ satisfying $\kappa^2<a-\varepsilon_R M-\delta$ consider
\begin{equation}\label{Z-eq1}
\underline{u}_t=\underline{u}_{xx} + \underline{u}(a-\varepsilon_R M-\delta- (b-\chi\mu)\underline{u})
\end{equation}
 and set $c^{\delta}_{\kappa}:=\frac{a-\varepsilon_R M-\delta +\kappa^2}{\kappa}$. Let $\underline{U}^{\delta}$ denote the monotone decreasing traveling wave solution of \eqref{Z-eq1} connecting $\underline{u}(t,x)\equiv 0$ and $\underline{u}(t,x)\equiv \frac{a-\varepsilon_R M-\delta}{b-\chi\mu}$ satisfying
$$
\lim_{x\to\infty}\frac{\underline{U}^{\delta}(x)}{e^{-\kappa x}}=1.
$$
Let $B_{\varepsilon}(t,x)=\min\{\varepsilon,-\chi v_x(t,x;u_0)\}$ for every $0<\varepsilon\ll 1$ and $\underline{u}^{\delta,\sigma}(t,x)=\sigma \underline{U}^{\delta}(x-(c_{\kappa}^{\delta}-\varepsilon)t)$ for every $0<\sigma\ll1$. Thus $\underline{u}^{\delta}(t,x)$ satisfies
\begin{align*}
\underline{u}^{\delta,\sigma}_t= &\underline{u}^{\delta,\sigma}_{xx}+\varepsilon\underline{u}^{\delta,\sigma}_x + \underline{u}^{\delta,\sigma}(a-\varepsilon_ R M-\delta-\frac{(b-\chi\mu)}{\sigma}\underline{u}^{\delta,\sigma})\cr
\le&\underline{u}^{\delta,\sigma}_{xx}+B_{\varepsilon}(t,x)\underline{u}^{\delta,\sigma}_x + \underline{u}^{\delta,\sigma}(a-\varepsilon_R M-\delta-\frac{(b-\chi\mu)}{\sigma}\underline{u}^{\delta,\sigma})\quad (\text{since } \,\, \underline{u}^{\delta,\sigma}_x\leq 0).
\end{align*}
Hence, since $\sup_{t\geq 0}\|\underline{u}^{\delta,\sigma}(t,\cdot)\|_{\infty}\leq \frac{\sigma(a-\varepsilon_R M-\delta)}{b-\chi\mu}\to0$ as $\sigma\to 0^+$, then there is $0<\sigma_0\ll 1$ such that
\begin{align}\label{Z-eq2}
\underline{u}^{\delta,\sigma}_t
\le & \underline{u}^{\delta,\sigma}_{xx}+B_{\varepsilon}(t,x)\underline{u}^{\delta,\sigma}_x + \underline{u}^{\delta,\sigma}(a-\varepsilon_R M-C_R\sqrt[p]{\underline{u}^{\delta,\sigma}}-(b-\chi\mu)\underline{u}^{\delta,\sigma})
\end{align}
for every $0<\sigma<\sigma_0$.

Choose $\eta_1>0$ satisfying
$$
C_R\sqrt[p]{\eta_1}+\eta_1+\varepsilon_RM<\frac{1}{2}\min\{\varepsilon,\underline{m}(m,\chi,\mu,a,b,\lambda,M)\},
$$
and  choose $m_0\gg 0$ such that
$$
Me^{-\kappa m_0}<\eta_1.
$$
Let $N>c_{\kappa}+1+m_0$ be fixed and set $$
m_1:=\frac{1}{3}\inf\{u(t,x)\, |\, 1\leq t\leq N, x\leq (c_{\kappa}+N)N+m_0+1\}.
$$

{  We claim that $m_1>0$. Indeed, suppose by contradiction that  $m_1=0$, then there exist a sequence $x_n \leq (c_{\kappa}+N)N+m_0+1$ and a sequence $1\leq t_n\leq N$ such that \begin{equation}\label{dd-eq1}
u(t_n,x_n;u_0)\to0 \,\, \text{ as}\,\, n\to\infty.
\end{equation} Since $\{t_n\}_{n\geq 1}$ is bounded, without loss of generality, we may suppose that it converges to some $t^*\in[1,N]$. Note that $x_n\to-\infty$ as $n\to\infty$, otherwise without loss of generality, we may suppose that $(t_n,x_n)\to (t^*,x^*)$.  So, $u(t_n,x_n;u_0)\to u(t^*,x^*;u_0)>0$, contradicting the choice of the sequence $(t_n,x_n)$. Now, set $u_{0n}(x):=u_0(x+x_n)$, and observe this is a sequence of uniformly bounded and equicontinuous functions. So, by Arzela-Ascoli's Theorem, it converges (up to a subsequence) locally uniformly  to some function $\tilde{u}_0\in C^b_{\rm unif}(\R)$. Furthermore, since $x_n\to-\infty$ as $n\to\infty$, it follows that $\inf_{x\in\R}\tilde{u}_0(x)\geq \inf_{x\le 0}u_0(x)>m$. Note also that $\|\tilde{u}_0\|_{\infty}\leq M$. Thus, by Theorem \ref{persitence-tm}, we have that
\begin{equation}\label{dd-eq2}
u(t,x;\tilde{u}_0)\geq \underline{m}(m,\chi,\mu,a,b,\lambda,M) \quad \forall\ x\in\R, t\ge 0.
\end{equation} But by \cite[Lemma 3.2]{SaSh6_1}, we have that $$
(u(t+t_n,x;u_{0n}),v(t+t_n,x;u_{0n}))\to (u(t+t^*,x;\tilde{u}_0),v(t+t^*,x;\tilde{u}_0)) \,\, \text{as}\,\ n\to\infty
$$
locally uniformly. In particular \begin{equation}\label{dd-eq3}
u(t_n,0;u_{0n}) \to u(t^*,0;\tilde{u}_0)  \,\, \text{as}\,\ n\to\infty.
\end{equation}
Noting that
$$
u(t_n,0;u_{0n})=u(t_n,0;u_{0}(\cdot+x_n))=u(t_n,x_n;u_{0})\,\,\forall\ n\ge 1,
$$
it follows from \eqref{dd-eq1}- \eqref{dd-eq2} that $\underline{m}(m,\chi,\mu,a,b,\lambda,M)=0$, which yields a contradiction. Thus $m_1>0$.

}

Chose $0<\sigma_1<\sigma_0$ satisfying
$$\frac{\sigma_1(a-\delta-\varepsilon_R M)}{b-\chi\mu}<\min\{\eta_1  ,m_1e^{-(a-\delta)N}\}.$$
We claim that
\begin{equation}\label{Z-eq4}
\underline{u}^{\delta,\sigma_1}(t,x)< u(t,x;u_0),\quad \forall\ x\leq (c_{k}+N)t +m_0, \  t\ge 1.
\end{equation}

Indeed, observe that $\|\underline{u}^{\delta,\sigma_1}(t,\cdot)\|_{\infty}<\eta_1$ for every $t\geq 1$ and
$$
\|\underline{u}^{\delta,\sigma_1}(t,\cdot)\|_{\infty}\leq \frac{\sigma_1(a-\delta-\varepsilon_R M)}{b-\chi\mu}e^{(a-\delta)t}<m_1\leq u(t,x;u_0),\quad \forall\ x\leq (c_{k}+N) N+m_0, \quad 1\le t\leq N.
$$
Thus, by \eqref{Z-eq0}, using similar arguments as those in Step 2 of the proof of \eqref{T1-eq1}, we conclude that \eqref{Z-eq4} holds.
By \eqref{Z-eq4}, we deduce that
$$
\liminf_{t\to\infty}\inf_{x\leq (c_\kappa -2\varepsilon)t}u(t,x;u_0)\geq \frac{(a-\delta)\sigma_1}{b-\chi\mu}>0.
$$
Whence, \eqref{Spreadind-speed-eq1-1} follows since $\varepsilon$ is arbitrary chosen.
\end{proof}

\medskip

\medskip

\begin{proof} [Proof of Theorem \ref{Main-tm-2}(2)]

Assume $b>2\chi\mu$. Using \eqref{Spreadind-speed-eq1-1}, the proof of \eqref{Spreadind-speed-eq2} follows similar arguments as the proof of \eqref{T1-eq2}.  So, it remains to prove that \eqref{Spreadind-speed-eq3} holds.

To this end, set  $c_{\kappa}=\frac{a+\kappa^2}{\kappa}$  and let $0\ll \varepsilon\ll 1$. Consider the set
\begin{align*}
\mathcal{E}^{T,\varepsilon}_{\kappa,D}:=\{u\in C([0,T]: C^b_{\rm unif}(\R))\, |\, & \ u(0,\cdot)=u_0(\cdot)\ \text{and}\\
 & 0\leq u(t,x)\leq (1+\varepsilon)\overline{U}_{ \kappa,D}(x)\ \forall\ x\in\R, \ 0\leq t\leq T\},
\end{align*}
 where $\overline{U}_{ \kappa,D}$  is defined in \eqref{l-003}.
For every $u\in\mathcal{E}^{T,\varepsilon}_{\kappa,D}$, let $\Phi(t,x;u)$ denote the solution of
\begin{equation*}
\begin{cases}
\Phi_t=\mathcal{A}_{u}(\Phi), \ x\in\R, 0<t\leq T,\cr
\Phi(0,x)=u_0(x),\quad x\in\R.
\end{cases}
\end{equation*}
where
$$
\mathcal{A}_{u}(\Phi):=\Phi_{xx}+(c_\kappa-\chi\Psi_x(\cdot,\cdot;u))\Phi_x+(a-\chi\lambda\Psi(\cdot,\cdot;u)-(b-\chi\mu)\Phi)\Phi.
$$
 Thus for every $u\in\mathcal{E}^{T,\varepsilon}_{\kappa,D}$, we have that
\begin{align*}
\mathcal{A}_{u}((1+\varepsilon)D)=\left(a-\chi\lambda\Psi(t,x;u)-(b-\chi\mu)D(1+\varepsilon) \right)D(1+\varepsilon)\leq 0,
\end{align*}
whenever  $D\geq \frac{a}{(b-\chi\mu) (1+\varepsilon)}$.
Hence, by comparison principle for parabolic equations, we have that
\begin{equation}\label{l-1}
\Phi(t,x;u)\leq  (1+\varepsilon)D, \quad\forall\ u\in\mathcal{E}^{T,\varepsilon}_{\kappa,D},\ x\in\R, t\geq 0.
\end{equation}
On the other hand, for any $\tilde \kappa$ with $0<\kappa<\tilde{\kappa}< \min\{\sqrt{a},\sqrt \lambda\}$,  using  \eqref{estimate-on-space-derivative-1}, we have that
\begin{equation}\label{l-2}
{\kappa}\Psi_{x}(\cdot,\cdot,u)-\lambda\Psi(\cdot,\cdot;u)\leq0 \quad \text{and}\quad \tilde{\kappa}\Psi_{x}(\cdot,\cdot,u)-\lambda\Psi(\cdot,\cdot;u)\le 0
\end{equation}
and
\begin{align*}
&\mathcal{A}_{u}((1+\varepsilon)\left(e^{-\kappa x}+De^{-\tilde{\kappa}x}\right))\cr
=&(1+\varepsilon)D\left(\tilde{\kappa}^2-\tilde{\kappa}c_{\kappa}+a\right)e^{-\tilde{\kappa}x}+\chi(1+\varepsilon)\left(\kappa\Psi_{x}(\cdot,\cdot,u)-\lambda\Psi(\cdot,\cdot;u)\right)e^{-\kappa x} \cr
&+ D\chi(1+\varepsilon)\left(\tilde{\kappa}\Psi_{x}(\cdot,\cdot,u)-\lambda\Psi(\cdot,\cdot;u)\right)e^{-\tilde{\kappa} x}-(b-\chi\mu)(1+\varepsilon)^2\left(e^{-\kappa x}+De^{-\tilde{\kappa} x}\right)^2\cr
\leq&  (1+\varepsilon)D\left(\tilde{\kappa}^2-\tilde{\kappa}c_{\kappa}+a\right)e^{-\tilde{\kappa}x}-(b-\chi\mu)(1+\varepsilon)^2\left(e^{-\kappa x}+De^{-\tilde{\kappa} x}\right)^2\cr
=&D(1+\varepsilon) (\frac{\tilde{\kappa}}{\kappa}-1)(\tilde{\kappa}\kappa-a)e^{-\tilde{\kappa}x }-(b-\chi\mu)(1+\varepsilon)^2\left(e^{-\kappa x}+De^{-\tilde{\kappa} x}\right)^2\cr
\leq& 0.
\end{align*}
 Whence, comparison principle for parabolic equations and \eqref{l-1}  yield  that
$$
\Phi(t,x;u)\leq (1+\varepsilon)(e^{-\kappa x}+De^{-\tilde{\kappa}x}),\quad\forall\ u\in\mathcal{E}^{T,\varepsilon}_{\kappa,D},\ x\in\R, t\geq 0.
$$
Therefore,
$$
\Phi(t,x;u)\leq (1+\varepsilon)\overline{U}_{\kappa,D}(x),\quad\forall\ u\in\mathcal{E}^{T,\varepsilon}_{\kappa,D},\ x\in\R, t\geq 0.
$$
Following the arguments of the proof of \cite[Theorem 3.1]{SaSh2}, it can be shown that the function  $\Phi \ :\mathcal{E}^{T,\varepsilon}_{\kappa,D}\ni u \mapsto \Phi(u)\in \mathcal{E}^{T,\varepsilon}_{\kappa,D} $ is continuous and compact in the open compact topology. Hence by Schauder's fixed point theorem there is $u^*\in \mathcal{E}^{T,\varepsilon}_{\kappa,D}$, such that $\Phi(u^*)=u^*$. Note that $(u^*(t,x- c_{\kappa}t),v^*(t,x- c_{\kappa}t))$ is also a solution of \eqref{Keller-Segel-eq}, with $u^*(0,x)=u_0(x)$ for every $x\in\R$. Hence, by uniqueness of the solution to \eqref{Keller-Segel-eq}, we conclude that
$$
u^*(t,x-c_{\kappa}t)=u(t,x;u_0),\quad \forall\ x\in\R, \ 0\leq t\leq T.
$$
Hence $u(t,x;u_0)\in \mathcal{E}^{T,\varepsilon}_{\kappa,D}$. Since $T$ was arbitrary chosen, it follows that
\begin{equation}\label{l-3}
u(t,x;u_0)\leq (1+\varepsilon)\overline{U}_{\kappa,D}(x-c_{\kappa}t), \forall\ x\in\R, \forall\ t\ge0.
\end{equation}

Next, for $x\in O_{\kappa}:=\{y\, |\, e^{-\kappa y}>De^{-\tilde{\kappa}y}\}$, taking $A_{\kappa}:=(1-\frac{\tilde{\kappa}}{\kappa})(\tilde{\kappa}\kappa-a)$, and using \eqref{l-2}, we have
that
$$
 D\chi(1-\varepsilon)\left(\lambda\Psi(\cdot,\cdot;u)-\tilde{\kappa}\Psi_{x}(\cdot,\cdot,u)\right)e^{-\tilde{\kappa} x}+(b-\chi\mu)(1-\varepsilon)^2D\left[2e^{-\kappa x}-De^{-\tilde{\kappa}x}\right]e^{-\tilde{\kappa} x}\geq 0.
$$
Whence for $x\in O_{\kappa}$, it holds that
\begin{align*}
&\mathcal{A}_{u}((1-\varepsilon)\left(e^{-\kappa x}-De^{-\tilde{\kappa}x}\right))\cr
=&(1-\varepsilon)DA_{\kappa}e^{-\tilde{\kappa}x}-(b-\chi\mu)(1-\varepsilon)^2e^{-2\kappa x}+\chi(1-\varepsilon)\left(\kappa\Psi_{x}(\cdot,\cdot,u)-\lambda\Psi(\cdot,\cdot;u)\right)e^{-\kappa x} \cr
&+ D\chi(1-\varepsilon)\left(\lambda\Psi(\cdot,\cdot;u)-\tilde{\kappa}\Psi_{x}(\cdot,\cdot,u)\right)e^{-\tilde{\kappa} x}+(b-\chi\mu)(1-\varepsilon)^2D\left[2e^{-\kappa x}-De^{-\tilde{\kappa}x}\right]e^{-\tilde{\kappa} x}\cr
\ge&(1-\varepsilon)DA_{\kappa}e^{-\tilde{\kappa}x}-(b-\chi\mu)(1-\varepsilon)^2e^{-2\kappa x}+\chi(1-\varepsilon)\left(\kappa\Psi_{x}(\cdot,\cdot,u)-\lambda\Psi(\cdot,\cdot;u)\right)e^{-\kappa x}\cr
\geq &(1-\varepsilon)DA_{\kappa}e^{-\tilde{\kappa}x}-(b-\chi\mu)(1-\varepsilon)^2e^{-2\kappa x}-\chi(1-\varepsilon)\left(\kappa|\Psi_{x}|(\cdot,\cdot,u)+\lambda\Psi(\cdot,\cdot;u)\right)e^{-\kappa x}.
\end{align*}
Since $ e^{-\kappa x}>De^{-\tilde{\kappa}x}$, it follows from \eqref{l-001} and \eqref{l-002}  that
$$
\kappa|\Psi_{x}|(\cdot,\cdot,u)+\lambda\Psi(\cdot,\cdot;u)\leq B_{\kappa,\tilde{\kappa}}e^{-\kappa x},
$$
where $B_{\kappa,\tilde{\kappa}}:=\mu\left(\frac{\kappa}{\sqrt{\lambda-\kappa^2}}+\frac{\kappa^2+\lambda}{\lambda-\kappa^2}+\frac{ D\kappa}{\sqrt{\lambda-\tilde{\kappa}^2}}+\frac{ D (\lambda+\tilde{\kappa}\kappa)}{\lambda-\tilde{\kappa}^2}\right)$. Hence for every $x\in O_{\kappa}$, we have
\begin{align*}
\frac{\mathcal{A}_{u}((1-\varepsilon)\left(e^{-\kappa x}-De^{-\tilde{\kappa}x}\right))}{(1-\varepsilon)}\geq \left( DA_{\kappa}-\left( (b-\chi\mu)(1-\varepsilon)-\chi B_{\kappa,\tilde{\kappa}}\right)e^{-(2\kappa-\tilde{\kappa}) x}\right)e^{-\tilde{\kappa}x}\geq 0
\end{align*}
whenever $D\ge \frac{(b-\chi\mu)(1-\varepsilon)-\chi B_{\kappa,\tilde{\kappa}}}{A_{\kappa}}$, since $x>0$ for every $x\in O_{\kappa}$. Therefore, since the $x\mapsto e^{-\kappa x}-De^{-\tilde{\kappa}x}$ equals zero on the boundary of $O_{\kappa}$, then by comparison principle for parabolic equations, we deduce that
 $$
(1-\varepsilon)\left(e^{-\kappa x}-De^{-\tilde{\kappa}x}\right)\leq \Phi(t,x;u),\quad\forall\ u\in\mathcal{E}^{T,\varepsilon}_{\kappa,D},\ x\in\R, t\geq 0.
$$
In particular, it follows that
$$
(1-\varepsilon)\left(e^{-\kappa x}-De^{-\tilde{\kappa}x}\right)\leq \Phi(t,x;u^*)=u(t,x+ c_\kappa t;u_0),\quad\ x\in\R, t\geq 0.
$$
It is clear from comparison principle that
$$\limsup_{t\to\infty}\|\lambda \chi v(t,\cdot;u_0)\|_{\infty}\leq\limsup_{t\to\infty}\|\mu \chi u(t,\cdot;u_0)\|_{\infty}\leq \frac{\chi\mu a}{b-\chi\mu}$$
Hence, since $2\chi\mu<b$,  we can choose $t_{\varepsilon}\gg 1$ such that
$$
a-\chi\lambda\|v(t,\cdot,u_0)\|_{\infty}\geq a-\frac{\chi\mu (1+\varepsilon)}{b-\chi\mu}=\frac{a(b-2\chi\mu-\chi\mu\varepsilon)}{b-\chi\mu}>0, \quad \forall\ t\geq t_{\varepsilon}.
$$
Thus,
$$
\mathcal{A}_{u^*}(\delta)\geq (a-\chi\lambda \|v(t,\cdot;u_0)-(b-\chi\mu)\delta)>0, \quad \forall\ 0<\delta\ll \frac{a(b-2\chi\mu-\chi\mu\varepsilon)}{(b-\chi\mu)^2}, \ t\geq t_{\varepsilon}.
$$
Observer from \eqref{l-005} that
$$
\lim_{D\to\infty}\max\{\varphi_{\kappa}(x)-D\varphi_{\tilde{\kappa}}(x)\, |\, x\in\R\}=0.
$$
 Therefore, it follows from \eqref{l-004} and comparison principle for parabolic equations that
\begin{equation}\label{l-4}
(1-\varepsilon)\underline{U}_{\kappa,D}(x)\leq \Phi(t,x;u^*)=u(t,x+ c_{\kappa} t;u_0),\quad\ x\in\R, t\geq t_{\varepsilon}.
\end{equation}
Which combined with  \eqref{l-3} yields that
$$
(1-\varepsilon)\underline{U}_{\kappa,D}(x-c_{\kappa}t)\leq u(t,x;u_0)\leq(1+\varepsilon)\overline{U}_{\kappa,D}(x-c_{\kappa}t), \forall\ x\in\R, \forall\ t\ge t_{\varepsilon}.
$$
This implies that
$$
\sup_{x\geq ( c_{\kappa}+\tilde{\varepsilon})t}\left|\frac{u(t,x;u_0)}{e^{-\kappa(x-c_{\kappa}t)}}-1\right|\leq \varepsilon+ (1+\varepsilon)D\sup_{x\geq ( c_{\kappa}+\tilde{\varepsilon})t}e^{-(\tilde{\kappa}-\kappa)(x-c_{\kappa}t)}\leq \varepsilon+D(1+\varepsilon)e^{-(\tilde{\kappa}-\kappa)\tilde{\varepsilon} t}.
$$
So \eqref{Spreadind-speed-eq3} follows since $\varepsilon$ was arbitrary chosen.
\end{proof}

\section{Traveling wave solutions}

In this section we study the existence of traveling wave solutions and prove the following result, which is an application of the results established in the previous section and the theory developed in \cite{SaSh2}. In order to make use of the theory established in \cite{SaSh2}, we first set up the right framework which follows from the proof of Theorem \ref{Main-tm-2}.

For every $0<\kappa<\tilde{\kappa}<\min\{\sqrt{a},\sqrt{\lambda}\}$ with $\tilde{\kappa}<2\kappa$,  $D \ge 1$, define
$$
U^+_{\kappa}(x)=\min\{\frac{a}{b-\chi\mu},e^{-\kappa x}\}
$$
and
$$
U^{-}_{\kappa,D}(x)=\max\{0, e^{-\kappa x}-De^{-\tilde{\kappa}x}\}
$$
and consider the set
$$
\mathcal{E}_{\kappa}=\{u\in C^b_{\rm unif}(\R)\ :\ 0\leq u(x)\leq U^+_{\kappa}(x)\ \forall\ x\in\R\}.
$$
For every $u\in\mathcal{E}_{\kappa}$, consider the operator
$$
\mathcal{A}_{u,\kappa}(U)=U_{xx}+(c_{\kappa}-\chi\Psi_x(x;u))U_{x}+(a-\chi\lambda\Psi(x;u)-(b-\chi\mu)U)U
$$
where $\Psi(x;u)$ is given by \eqref{psi-definition}. Consider the function
$$
U(x;u)=\limsup_{t\to\infty}U(t,x;u),\quad \forall\ u\in\mathcal{E}_{\kappa},
$$
where $U(t,x;u)$ is the solution of the parabolic equation
\begin{equation}\label{U-def}
\begin{cases}
U_t=\mathcal{A}_{u,\kappa}(U),\quad\ t>0, x\in\R,\cr
U(0,x;u)=U^+_{\kappa}(x),
\end{cases}
\end{equation}

It follows from the arguments used in the proof of Theorem \ref{Main-tm-2} that $U^+_{\kappa}$ is supper-solution for \eqref{U-def}, hence comparison principle for parabolic equations imply that
$$
U(t_2,x;u)\leq U(t_1,x;u)\leq U^+_{\kappa}(x),\quad \forall x\in\R,\ 0<t_1<t_2, \ u\in\mathcal{E}_{\kappa}.
$$
Thus
\begin{equation}\label{U-main-def}
U(x;u)=\lim_{t\to\infty}U(t,x;u),\quad \forall\ u\in\mathcal{E}_{\kappa}.
\end{equation}
Moreover, using estimates for parabolic equation, one can show that $U(x;u)$ satisfies the elliptic equation
\begin{equation}\label{U-elliptic-eq}
0=U_{xx} +(c_{\kappa} -\chi\Psi_x(x;u))U_x+(a-\chi \lambda\Psi(x;u)-(b-\chi\mu)U)U.
\end{equation}
On the other hand, it follows also from the argument used in the proof of Theorem \ref{Main-tm-2} that there is $D\gg 1$ such that  $U^{-}_{\kappa,D}(x)$ is a subsolution of \eqref{U-def}. Whence,
\begin{equation}\label{lower-bound of U}
U^-_{\kappa,D}(x)\leq U(t,x;u),\quad \quad \forall x\in\R,\ t>0, \ u\in\mathcal{E}_{\kappa}.
\end{equation}
Hence,
\begin{equation}
U^-_{\kappa,D}(x)\leq U(x;u),\quad \forall\ x\in\R,\ \forall\ u\in\mathcal{E}_{\kappa}.
\end{equation}

With these setting we can now apply the theory developed in \cite{SaSh2}.

\begin{proof} [Proof of Theorem \ref{Existence of traveling}]

  (1) Consider the mapping  $U(\cdot;\cdot): \mathcal{E}_{\kappa}\ni u \mapsto U(x;u)\in\mathcal{E}_{\kappa}$ as defined by \eqref{U-main-def}. It follows from the arguments of the proof of \cite[Theorem  3.1 ]{SaSh2} that this function is continuous and compact in the compact open topology. Hence it has a fixed point $u^*$ by the Schauder's fixed point Theorem. Taking $v^*(x)=\Psi(x;u^*)$, we have from \eqref{U-elliptic-eq}, that $(u(x,t),v(x,t))=(u^*(x-c_{\kappa}t),v^*(x-c_{\kappa}t))$ is an entire solution of \eqref{Keller-Segel-eq}. Moreover, since $U^{-}_{\kappa,D}(x)\leq u^*(x)\leq U^+_{\kappa}(x)$ it follows that
$$
\lim_{x\to\infty}\frac{u^*(x)}{e^{-\kappa x}}=1.
$$
Note from \eqref{lower-bound of U} that $u^*(x)>0$ for every $x\in\R$. Therefore, it follows from \cite[Theorem 1.1]{FhCh} that

$$
\liminf_{t\to\infty}\inf_{|x|\leq \sqrt{a}t}u^*(x-ct)>0.
$$
Which implies that
$$\liminf_{x\to-\infty}u^*(x)>0, $$
this completes the proof of  \eqref{nnew-eq0}.

 Suppose now that $\chi\mu<\frac{b}{2}$. Let $(u(t,x),v(t,x))=(U(x-c_{\kappa}t),V(x-c_{\kappa}t))$ be a nontrivial traveling wave   established in the above.  We claim that
\begin{equation}\label{kk-0}
\lim_{x\to-\infty}|U(x)-\frac{a}{b}|=0.
\end{equation}
Suppose on the contrary that this is false. Then, there is $x_{n}\to-\infty$ such that
\begin{equation}\label{kkk-1}
\inf_{n\geq 1}|U(x_n)-\frac{a}{b}|>0.
\end{equation}
Consider the sequence of functions
$$
u^{n}(t,x)=u(t,x+x_n)\quad \text{and}\quad v^{n}(t,x)=v(t,x+x_n).
$$
By a priori estimate for parabolic equation, without loss of generality, we suppose that there is $(u^*(t,x),v^*(t,x))\in C^{1,2}(\R\times\R)$ such that $(u^{n},v^{n})(t,x)\to$  $(u^*(t,x),v^*(t,x))$ as $n\to\infty$. Furthermore, the function is an entire solution of \eqref{Keller-Segel-eq}. Note that
$$
0<\liminf_{x\to-\infty}U(x)\leq u^*(t,x)\leq \frac{a}{b-\chi\mu},\quad \forall\ x\in\R, \ t\in\R.
$$
Therefore, since $\chi\mu<\frac{b}{2}$, it follows from the stability of the constant positive $(\frac{a}{b},\frac{\mu a}{\lambda b})$ of \eqref{Keller-Segel-eq} (see \cite[Theorem 1.8]{SaSh1}) that $u^*(t,x)=\frac{a}{b}$ for every $x,t\in\R$. In particular, $u^*(0,0)=\frac{a}{b}$, which contradicts to \eqref{kkk-1}. Therefore, \eqref{kk-0} must hold.

\smallskip

(2)  Suppose that $\chi\mu<\frac{b}{2}.$
 For every $c_n>c^{**}$ with $c_n\to  c^{**}$, let $(U^{c_n}(x),V^{c_n}(x))$ be the traveling wave solution of \eqref{Keller-Segel-eq} connecting $(0,0)$ and $(\frac{a}{b},\frac{\mu a}{\lambda b})$ with speed $c_n$ given by Theorem \ref{Existence of traveling} (1). Note from the proof of existence of $(U^{c_n}(x),V^{c_n}(x))$ that
 $$
 U^{c_n}(x)\leq \min\{\frac{a}{b-\chi\mu},e^{-\kappa_n x}\},\quad\ x\in\R,  $$ where $\kappa_n=\frac{c_n-\sqrt{c_n^2-4 a}}{2}$. For each $n\geq 1$, note that the set $\{x\in\R\, |\, U^{c_n}(x)=\frac{a}{2b}\}$ is compact  and nonempty, so there is $x_n\in\R$ such that
 $$x_n=\min\{x\in\R\, |\, U^{c_n}(x)=\frac{a}{2b}\}.$$  Since $\|U^{c_n}\|_{\infty}<\frac{a}{b-\chi\mu}$ for every $n\geq1$, hence by estimates for parabolic equations, without loss of generality, we may suppose that $U^{c_n}(x+x_n)\to U^*(x)$ as $n\to\infty$ locally uniformly. Furthermore, taking $V^*=\Psi(x;U^*)$, it holds that $(U^*,V^*)$ solves
 \begin{equation}\label{kk-2}
 \begin{cases}
0=U^*_{xx}+( c^{**}-\chi V^*_x)U^*_x+U^*(a-\chi\lambda V^*-(b-\chi\mu)U^*),\quad x\in\R\cr
0=V^*_{xx}-\lambda V^*+\mu U^*,\quad x\in\R,
\end{cases}
 \end{equation}
 $U^*(0)=\frac{a}{2b}$ and $U^*(x)\geq \frac{a}{2b}$ for every $x\leq 0$. Next we claim that
 \begin{equation}\label{kk-3}
\overline{U}^*(\infty):= \limsup_{x\to\infty}U^*(x)=0.
 \end{equation}
 Suppose on the contrary that \eqref{kk-3} does not hold. Then there is a sequence $\{x_n\}_{n\geq1}$ such that $x_n<x_{n+1}$ for every $n$, with $x_1=0$,  $x_n\to\infty$ and
 \begin{equation*}
 U^*(x_n)\geq \frac{\overline{U}^*(\infty)}{2}>0,\quad\forall n\geq 1.
 \end{equation*}
 For every $n\geq 1$ let $\{y_n\}_{n\geq1}$ be the sequence defined by
 $$
U^*(y_n)=\min\{U^*(x)\, |\, x_n\leq x\leq x_{n+1}\}.
 $$
 It is clear that
 $$
\lim_{n\to\infty}U^*(y_n)=\inf_{x\in\R}U^*(x).
 $$
 Since $(U^*(x- c^{**}t),V^*(x- c^{**}t))$ is a positive entire solution of \eqref{Keller-Segel-eq} with
 $$
\liminf_{x\to-\infty}U^*(x) \geq \frac{a}{2b}\quad \text{and} \quad U^*(0)\neq \frac{a}{b},
 $$
 then by the stability of the constant equilibrium $(\frac{a}{b},\frac{\mu a}{\lambda b})$ (see \cite[Theorem 1.8]{SaSh1}), we obtain that
 $$
0=\inf_{x\in\R}U^*(x)=\lim_{n\to\infty}U^*(y_n).
 $$
 Therefore, without loss of generality, we may suppose that $y_n\in(x_n,x_{n+1})$ for every $n\geq 1$ with
 $$
\frac{d^2}{dx^2}U^*(y_n)\geq 0 \quad \text{and}\quad  \frac{d}{dx}U^*(y_n)=0, \quad \forall\ n\geq1.
 $$
 Note from \eqref{E2-1} that we also have that
 $$
\lim_{n\to\infty}V^*(y_n)=0.
 $$
 Thus for $n$ large enough, we have that
 $$
 U_{xx}^*(y_n)-(c-\chi V^*(y_n))U^*_x(y_n)+U^*(y_n)(a-\chi\lambda V^*(y_n)-(b-\chi\mu)U^*(y_n))>0
 $$
 which contradicts to \eqref{kk-2}. Therefore, \eqref{kk-3} holds.

 It follows again from
 $$
\liminf_{x\to-\infty}U^*(x) \geq \frac{a}{2b} $$  and the stability of the constant equilibrium $(\frac{a}{b},\frac{\mu a}{\lambda b})$ that
$$ \lim_{x\to-\infty}U^*(x)=\frac{a}{b}.$$
 Therefore $(u(t,x),v(t,x))=(U^*(x-c^{**}t),V^*(x-c^{**}t))$ is a traveling wave solution of \eqref{Keller-Segel-eq} with speed $c^{**}$ connecting $(0,0)$ and $(\frac{a}{b},\frac{\mu a}{\lambda b})$.

\smallskip

(3)  Let $(u(t,x),v(t,x))=(U(x-ct),V(x-ct))$ be a nontrivial  traveling wave solution of \eqref{Keller-Segel-eq} with speed $c$.  Observing in the proof of Theorem \ref{Main-tm-2}(1), by taking $\underline{U}^{\delta}$ to be a  traveling wave solution of \eqref{Z-eq1} with speed $c_{\kappa}=2\sqrt{a-\varepsilon_R M-\delta}$ connecting $0$ and $\frac{a-\varepsilon_R M-\delta}{b-\chi\mu}$ and chose $N\gg 1$ such that $2\sqrt{a-\varepsilon_R M-\delta}+N>c$, so that
$$
\lim_{t\to\infty}\sup_{x\geq (2\sqrt{a-\varepsilon_RM-\delta}+N)t}u(t,x)=0,
$$ it follows from the arguments used there, that  for every $0<\varepsilon\ll 1$,
$$
0<\liminf_{t\to\infty}\inf_{x\leq (2\sqrt{a}-2\varepsilon)t}u(t,x)=\liminf_{t\to\infty}\inf_{x\leq (2\sqrt{a}-2\varepsilon)t}U(x-ct)\leq \liminf_{t\to\infty}U((2\sqrt{a}-2\varepsilon-c) t).
$$

Hence, since $U(\infty)=0$, we must have that $2\sqrt{a}- 2 \varepsilon\leq c$ for every $0<\varepsilon\ll 1$. Letting $\varepsilon \to 0$, we obtain that $c\ge 2\sqrt{a}$.
\end{proof}


\begin{thebibliography}{9}





\bibitem{AiHuWa} S. Ai, W. Huang, and Z.-A.  Wang,  Reaction, diffusion and chemotaxis in wave propagation,
 {\it Discrete Contin. Dyn. Syst. Ser. B} {\bf  20} (2015), no. 1, 1-21.

 \bibitem{AiWa} S. Ai and Z.-A.  Wang,
 Traveling bands for the Keller-Segel model with population growth,
 {\it  Math. Biosci. Eng.} {\bf 12} (2015), no. 4, 717-737.


\bibitem{ArWe2} { D. G. Aronson and H. F. Weinberger}, {  Multidimensional nonlinear diffusions arising in population genetics},  {\it Adv. Math.}, {\bf 30} (1978), pp. 33-76.



\bibitem{BBTW}
  N. Bellomo, A. Bellouquid,  Y. Tao, and M. Winkler,    Toward a mathematical theory of Keller-Segel models of pattern formation in biological tissues,  {\it  Math.~Models Methods Appl.~Sci.}, {\bf 25} (2015), 1663-1763.


\bibitem{BHN} H. Berestycki, F. Hamel and G. Nadin, Asymptotic spreding in heterogeneous diffusive excitable media, {\it J. Funct. Anal.}, {\bf 255} (2008), 2146-2189.
    
  \bibitem{BeHaNa1} H. Berestycki, F. Hamel, and N. Nadirashvili, The
speed of propagation for KPP type problems, I - Periodic framework, {\it J. Eur. Math. Soc.}, {\bf 7} (2005),
172-213.

\bibitem{BeHaNa2} H. Berestycki, F. Hamel, and N. Nadirashvili, The
speed of propagation for KPP type problems, II - General domains,  {\it J. Amer. Math. Soc.}, {\bf  23}  (2010),  no. 1,  1-34.
  


\bibitem{Henri1} H. Berestycki and G. Nadin, Asymptotic spreading for general heterogeneous Fisher-KPP type, preprint.


\bibitem{Fisher} { R. Fisher}, {The wave of advance of advantageous genes},
{\it Ann. of Eugenics}, {\bf 7} (1937), 335-369.



\bibitem{Fre} M. Freidlin, On wave front propagation in periodic media.
{\it In: Stochastic analysis and applications, ed. M. Pinsky, Advances in probablity and
related topics}, 7:147-166, 1984.

\bibitem{FrGa} M. Freidlin and J. G\" artner, On the propagation of
concentration waves in periodic and ramdom media, {\it Soviet Math. Dokl.}, {\bf 20} (1979),  1282-1286.





\bibitem{FuMiTs} M. Funaki, M. Mimura and T. Tsujikawa, Travelling front solutions arising in the chemotaxis-growth model,
 {\it Interfaces Free Bound.}, {\bf 8} (2006), 223-245.


\bibitem{FhCh} F. Hamel and C. Henderson, Propagation in a Fisher-KPP equation with non-local advection, preprint.

\bibitem{HiPa} T. Hillen and  K.J. Painter, A User’s Guide to PDE Models for Chemotaxis, {\it J. Math. Biol.} {\bf 58} (2009) (1), 183-217.


\bibitem{Hor} D. Horstmann,
From 1970 until present: the Keller-Segel model in chemotaxis and its consequences,
{\it Jahresber. Dtsch. Math.-Ver.}, {\bf 105} (2003),  103-165.


\bibitem{HoSt} D. Horstmann and A. Stevens, A constructive approach to traveling waves in chemotaxis, {\it  J. Nonlin. Sci.}, {\bf 14} (2004), 1-25.



\bibitem{DirkandWinkler} D. Horstmann and M. Winkler, Boundedness vs. blow up in a chemotaxis system,   {\it J. Differential Equations}, {\bf 215} (2005), 52-107.

\bibitem{ITBWS16} T. B. Issa and  W. Shen,  Dynamics in chemotaxis models of parabolic-elliptic type on bounded domain with time and space dependent logistic sources, {\it  SIAM J. Appl. Dyn. Syst.} {\bf 16} (2017), no. 2, 926-973.

\bibitem{KKAS}  K. Kang and A. Steven, Blowup and global solutions in a chemotaxis-growth system, {\it Nonlinear Analysis}, {\bf 135} (2016), 57-72.

\bibitem{kuto_PHYSD}   K. Kuto, K. Osaki,  T. Sakurai, and T. Tsujikawa,   Spatial pattern formation in a chemotaxis-diffusion-growth model,  {\it  Physica D}, {\bf 241} (2012), 1629-1639.

\bibitem{KeSe1} E.F. Keller and L.A. Segel, Initiation of slime mold aggregation viewed as an instability, {\it J. Theoret. Biol.},
{\bf 26} (1970), 399-415.

\bibitem{KeSe2}  E.F. Keller and L.A. Segel,  A  Model for chemotaxis, {\it  J. Theoret. Biol.}, {\bf 30} (1971),  225-234.

\bibitem{KPP} {A. Kolmogorov, I. Petrowsky, and N. Piskunov},
{A study of the equation of diffusion with increase in the quantity
of matter, and its application to a biological problem},
{\it  Bjul. Moskovskogo Gos. Univ.}, {\bf 1} (1937), pp. 1-26.

\bibitem{LiLiWa} J. Li, T. Li, and Z.-A. Wang,
 Stability of traveling waves of the Keller-Segel system with logarithmic sensitivity,
 {\it Math. Models Methods Appl. Sci.} {\bf 24} (2014), no. 14, 2819-2849.

\bibitem{LiZh} X. Liang and X.-Q. Zhao,
 Asymptotic speeds of spread and traveling waves for monotone semiflows
  with applications,  {\it Comm. Pure Appl. Math.},  {\bf 60}  (2007),  no. 1,
   1-40.
\bibitem{LiZh1} X. Liang and X.-Q. Zhao, Spreading speeds and traveling waves for abstract
monostable evolution systems, {\it Journal of Functional Analysis},  {\bf 259} (2010),  857-903.


\bibitem{LSV1} S. Luckhaus, Y. Sugiyama, J.J.L. Vel\"azquez, : Measure valued solutions of the 2D KellerSegel system. {\it Arch. Rat. Mech. Anal.} {\bf 206}, 31-80 (2012).
    
    \bibitem{MaNoSh} B. P. Marchant, J. Norbury, and J. A. Sherratt, Travelling wave solutions to a
haptotaxis-dominated model of malignant invasion, {\it Nonlinearity}, {\bf 14} (2001),  1653-1671.

\bibitem{AM_AK_LE} A. M\"uller-Taubenberger, A. Hortholt, and L. Eichinger, {Simple system - substantial share: The use of Dictyostelium in cell biology and
molecular medicine.} Eur. J. Cell Biol. {\bf (2012)}.

\bibitem{Nad} G. Nadin, Traveling fronts in space-time periodic media,
{\it J. Math. Pures Anal.}, {\bf 92} (2009),  232-262.



\bibitem{NaPeRy} G. Nadin, B. Perthame, and L. Ryzhik,
 Traveling waves for the Keller-Segel system with Fisher birth terms,
 {\it  Interfaces Free Bound.} {\bf 10} (2008), no. 4, 517-538.
 

\bibitem{NAGAI_SENBA_YOSHIDA} T. Nagai, T. Senba and K, Yoshida, Application of the Trudinger-Moser Inequality  to a Parabolic System of Chemotaxis, {\it Funkcialaj Ekvacioj}, {\bf 40} (1997), 411-433.
    


\bibitem{NoRuXi} J. Nolen, M. Rudd, and J. Xin, Existence of KPP
fronts in spatially-temporally periodic adevction and variational principle for propagation speeds, {\it
Dynamics of PDE}, {\bf 2} (2005),  1-24.

\bibitem{NoXi1} J. Nolen and J. Xin, Existence of KPP type fronts in
space-time periodic shear flows and a study of minimal speeds based on variational principle, {\it Discrete and
Continuous Dynamical Systems}, {\bf 13} (2005),  1217-1234.


\bibitem{SaSh6_1} R. B. Salako and W. Shen, Parabolic-elliptic chemotaxis model with space-time dependent logistic sources on $\mathbb{R}^N$. I. Persistence and asymptotic spreading. {\it Mathematical Models and Methods in Applied Sciences} Vol. {\bf 28}, No. 11, pp. 2237-2273 (2018)


 \bibitem{SaSh3} R. B. Salako and Wenxian Shen, Existence of Traveling wave solution of parabolic-parabolic chemotaxis systems, {\it  Nonlinear Analysis: Real World Applications}
Volume {\bf 42}, (2018), {93-119}.

\bibitem{SaSh2}  R. B. Salako and Wenxian Shen, Spreading Speeds and Traveling waves of a parabolic-elliptic chemotaxis system with logistic source on $\mathbb{R}^N$, {\it Discrete and Continuous Dynamical Systems - Series A},  {\bf 37} (2017), pp. {6189-6225}.


\bibitem{SaSh1} R. B. Salako and Wenxian Shen, Global existence and asymptotic behavior of classical solutions to a parabolic-elliptic chemotaxis system with logistic source on $\mathbb{R}^N$, {\it J. Differential Equations}, {\bf 262} (2017) 5635-5690.
    
  \bibitem{She1} W. Shen, Variational principle for spatial spreading
speeds and generalized propgating speeds  in time almost and space periodic KPP models, {\it Trans. Amer. Math.
Soc.}, {\bf 362} (2010),  5125-5168.

\bibitem{She2} W. Shen, Existence of generalized traveling waves in time recurrent and space periodic monostable equations,
{\it  J. Appl. Anal. Comput.}, {\bf  1} (2011),  69-93.
  
    

\bibitem{TeWi1} J.I. Tello and  M. Winkler,  Reduction of critical mass in a chemotaxis system by external application of a chemoattractant. {\it Ann. Sc. Norm. Sup. Pisa Cl. Sci.} {\bf 12} (2013) 833-862.

\bibitem{TeWi2} J. I. Tello and M. Winkler, A Chemotaxis System with Logistic Source,  {\it Communications in Partial Differential Equations}, {\bf 32} (2007), 849-877.
    
    \bibitem{Wan} Z.-A. Wang,  Mathematics of traveling waves in chemotaxis—review paper,
{\it  Discrete Contin. Dyn. Syst. Ser. B} {\bf 18} (2013), no. 3, 601-641.

\bibitem{Wei1} H. F.  Weinberger, Long-time behavior of a class of
biology models, {\it SIAM J. Math. Anal.}, {\bf 13} (1982),  353-396.

\bibitem{Wei2}H. F.  Weinberger, On spreading speeds and traveling
waves for growth and migration models in a periodic habitat, {\it J. Math. Biol.}, {\bf 45} (2002),  511-548.


\bibitem{win_jde}    M. Winkler,  Aggregation vs.~global diffusive behavior in the higher-dimensional Keller-Segel model,    {\it Journal of Differential Equations},  {\bf 248} (2010), 2889-2905.

\bibitem{win_JMAA_veryweak}    M. Winkler,  Blow-up in a higher-dimensional chemotaxis system despite logistic growth restriction,    {\it Journal of Mathematical Analysis and Applications}, {\bf 384} (2011), 261-272.

\bibitem{win_arxiv}   M. Winkler, Finite-time blow-up in the higher-dimensional parabolic-parabolic Keller-Segel system,   {\it J.~Math.~Pures Appl.},  {\bf 100} (2013), 748-767.

\bibitem{Zla} A. Zlato\v s,  Transition fronts in inhomogeneous Fisher-KPP reaction-diffusion equations,
 {\it J. Math. Pures Appl.} (9) {\bf 98} (2012), no. 1, 89-102.



\end{thebibliography}
\end{document}